\documentclass[a4paper,11pt,french,english]{amsart}
\usepackage[english]{babel}
\usepackage[utf8]{inputenc}
\usepackage{lmodern}
\usepackage[T1]{fontenc}
\usepackage{amsmath,amsthm,amssymb,amsfonts,mathrsfs}
\usepackage{amscd}
\usepackage{pstricks}
\usepackage{pst-node}
\usepackage[a4paper,left=3.2cm,right=3.2cm,top=4cm,bottom=4cm]{geometry}
\usepackage{enumerate}
\usepackage[shortlabels]{enumitem}
\usepackage[linktocpage=true]{hyperref}
\usepackage{url}
\usepackage{csquotes}
\usepackage{xspace}

\newcommand{\mc}{\mathcal}

\newcommand{\Aut}{\mathrm{Aut}}
\newcommand{\Comm}{\mathrm{Comm}}

\newcommand{\Res}{\mathrm{Res}}

\newcommand{\Z}{\mathbb{Z}}
\newcommand{\Q}{\mathbb{Q}}

\newcommand{\RadLE}{\mathrm{Rad_{LE}}}

\newcommand{\res}{\mathrm{Res}}

\newcommand{\Sub}{\ensuremath{\operatorname{Sub}}}%

\newcommand{\Core}{\ensuremath{\mathrm{Core}}}%
\newcommand{\QZ}{\ensuremath{\mathrm{QZ}}}%
\newcommand{\lpc}{\ensuremath{\pi_{\mathrm{loc}}}}%

\title[]{Lattices determined by their commensurator}

\date{April 6, 2026}

\author{Adrien Le Boudec}
\address{CNRS, UMPA - ENS Lyon, 46 all\'ee d'Italie, 69364 Lyon, France}
\email{adrien.le-boudec@ens-lyon.fr}

\author{Colin Reid}
\address{Western Sydney University, School of Computer, Data and Mathematical Sciences, Penrith, NSW 2751, Australia}
\email{colin@reidit.net}

\thanks{}

\theoremstyle{plain}
\newtheorem{Theorem}{Theorem}[section]
\newtheorem{Proposition}[Theorem]{Proposition}
\newtheorem{Corollary}[Theorem]{Corollary}
\newtheorem{Lemma}[Theorem]{Lemma}

\newtheorem{TheoremIntro}{Theorem}
    
\newtheorem{Corollary-intro}[TheoremIntro]{Corollary}

\theoremstyle{definition}
\newtheorem{Definition}[Theorem]{Definition}
\newtheorem{Question}{Question}

\newtheorem{Example-intro}{Example}
\newtheorem{Remark}[Theorem]{Remark}

\begin{document}

\maketitle

\begin{abstract}
	Let $\Gamma$ be a finitely generated cocompact lattice of a totally disconnected locally compact group $G$, and $C$ a dense subgroup of $G$ that contains and commensurates $\Gamma$. We study the problem of describing all finitely generated commensurated subgroups of $C$. We establish general rigidity results ensuring every finitely generated commensurated subgroup of $C$ is virtually contained in $\Gamma$. In more concrete situations, in fact we conclude that up to commensurability, $\Gamma$ is the only infinite finitely generated commensurated subgroup of $C$. For instance this last conclusion holds when $G$ is the automorphism group of a tree. This settles in particular the problem whether two non-commensurable cocompact tree lattices may have the same commensurator. Further applications include commensurators of cocompact lattices in other groups of automorphisms of trees, as well as commensurators of graph product of finite groups in automorphism groups of right-angled building.
\end{abstract}


\section{Introduction}

Let $G$ be a locally compact group, and $\Gamma$ a discrete subgroup of $G$. The commensurator $\Comm_G(\Gamma)$ of $\Gamma$ in $G$ is the set of $g \in G$ such that $\Gamma$ and $g \Gamma g^{-1}$ are commensurable. Recall that two subgroups are commensurable if their intersection has finite index in each of them. A general problem consists in relating properties of $\Gamma$ with those of its commensurator $\Comm_G(\Gamma)$. One key result in this realm is Margulis' arithmeticity criterion, which asserts that when $G$ is a connected semisimple Lie group with trivial center and no compact factor and $\Gamma$ is an irreducible lattice, then $\Gamma$ is arithmetic if and only if $\Comm_G(\Gamma)$ is a dense subgroup of $G$ \cite[Chap.\ IX (1.9)]{Margulis-book}. Another closely related result of Margulis in this setting is the commensurator superrigidity theorem, which asserts that a Zariski dense linear representation of $\Comm_G(\Gamma)$ in a simple algebraic group under which the image of $\Gamma$ is unbounded extends to a continuous representation of $G$ \cite[Chap.\ VII Th.\ 5.4]{Margulis-book}. 

Those results and the wealth of techniques employed had a vast influence in several other contexts, notably in situations where the ambient group $G$ is no longer a connected Lie group or an algebraic group over a local field, especially in the general setting of non-positively curved metric spaces and their isometry groups. One instance of such a situation is where $G = \Aut(T)$ is the automorphism group of a  locally finite tree. A remarkable fact in that setting is that when $\Aut(T)$ acts cocompactly on $T$, any two cocompact lattices of $\Aut(T)$ are always conjugate up to commensurability \cite{Leighton-finite-cover,Bass-Kulk-90}. In particular any two cocompact lattices of $\Aut(T)$ have their commensurator that are conjugate. Bass--Kulkarni for bi-regular trees and Liu in general showed that the commensurator of a cocompact lattice of $\Aut(T)$ is always dense \cite{Bass-Kulk-90, Liu-density}. Mozes \cite{Mozes-SL2-dense-commensurator} and Abramenko-Rémy \cite{Abramenko-Remy} showed the same result for certain non-cocompact lattices. Still in the context of cocompact lattices in automorphism groups of trees, Lubotzky--Mozes--Zimmer showed a commensurator superrigidity theorem (relative to the context, i.e.\ with target a group of automorphisms of a tree instead of a simple algebraic group) \cite{LMZ-superrigidity}. A vast generalization was then obtained by Burger--Mozes in certain $\mathrm{CAT}(-1)$ spaces \cite{BM-CAT-minus-one-commensurator}, and very general versions with only a non-positive curvature assumption made at the level of target groups were proven by Monod \cite{Monod-super-rigidity} and Gelander--Karlsson--Margulis \cite{Gelander-Karlsson-Margulis}. Shalom also established various rigidity results for lattices and their commensurators in a very general setting based on superrigidity for unitary representations \cite{Shalom-rigid-commens}. Beyond the case where $\Gamma$ is assumed a priori to be a lattice, discrete subgroups with a dense commensurator have been investigated by Leininger--Long--Reid \cite{Leininger-Long-Reid} and Mj \cite{Mj-discrete-comm-2011} when $G = \mathrm{PSL}(2,\mathbb{C})$, and more recently by Fisher--Mj--van Limbeek \cite{Fisher-Mj-vanLimbeek-JEP}.

In the present paper we are concerned with the general problem of understanding when a lattice $\Gamma$ in a locally compact group $G$ is determined by its commensurator.

\begin{Question} \label{quest-intro}
	Let $G$ be a locally compact group, and $\Gamma$ a lattice. Suppose $\Gamma'$ is a lattice in $G$ such that   $\Comm_G(\Gamma) =  \Comm_G(\Gamma')$. Is $\Gamma'$ necessarily commensurable with $\Gamma$ ?
\end{Question}

Any subgroup of $G$ commensurable with $\Gamma$ is also a lattice in $G$ with the same commensurator as $\Gamma$, so recovering $\Gamma$ up to commensurability is the best one can hope for. When $\Comm_G(\Gamma)$ is a discrete subgroup of $G$, then necessarily $\Gamma$ has finite index in its commensurator, and the answer is always positive. Hence Question \ref{quest-intro} is interesting for lattices with a non-discrete commensurator. Actually it is natural to assume $\Comm_G(\Gamma)$ is dense in $G$, as the natural locally compact group that encapsulates the context is the closure of $\Comm_G(\Gamma)$ in $G$. When $G$ is a connected semisimple Lie group with trivial center and no compact factor, the aforementioned arithmeticity criterion and superrigidity theorem imply a positive answer to Question \ref{quest-intro}. In the case where $G = \Aut(T)$ and the lattices $\Gamma, \Gamma'$ are cocompact, Question \ref{quest-intro} is a long-standing open problem. It was originally asked in \cite{Bass-Kulk-90}, and also appeared in \cite{LMZ-superrigidity, Mozes-SL2-dense-commensurator,Bass-Lub-book,Farb-Mozes-Thomas}. 

We will be working in the context where $G$ is a totally disconnected locally compact (tdlc) group, and lattices are finitely generated and cocompact. In that situation we actually consider a  more general problem than Question \ref{quest-intro}. First we consider the situation where $C$ is any dense subgroup of $G$ such that $\Gamma \leq C \leq \Comm_G(\Gamma)$. Second, and more significantly, we consider the problem of studying all finitely generated commensurated subgroups $H$ of $C$ (without assuming $H$ is a lattice in $G$). The group $C$ comes with a finitely generated commensurated subgroup given by the context (namely $\Gamma$), and the first question one can ask is whether there are other ones. That problem shares some similarities with the $S$-arithmetic version of the Margulis--Zimmer problem \cite{Sha-Willis}.

\bigskip

\textbf{Results.} In this general setting with $\Gamma$ a finitely generated cocompact lattice of a tdlc group $G$, and $C$ a dense subgroup of $G$ such that $\Gamma \leq C \leq \Comm_G(\Gamma)$, our first theorem provides sufficient conditions ensuring every finitely generated commensurated subgroup $H$ of $C$ is virtually contained in $\Gamma$ (i.e.\ some finite index subgroup of $H$ is contained in $\Gamma$). It is worth pointing out there is another natural commensurated subgroup of $C$ other than $\Gamma$, namely the intersection $C \cap U$ between $C$ and a compact open subgroup $U$ of $G$. Since $\Gamma \cap U$ is finite and $C \cap U$ is dense in $U$, the subgroup $C \cap U$ is never virtually contained in $\Gamma$ (provided we are not in the trivial situation where $G$ is discrete). Hence the conclusion of the theorem cannot encompass all commensurated subgroups of $C$. 

We recall some terminology. A profinite group is finitely generated if it has a dense finitely generated subgroup. A tdlc group $G$ is \textbf{locally finitely generated} if some compact open subgroup of $G$ is finitely generated (since all compact open subgroups are commensurable, this is equivalent to ask that every compact open subgroup is finitely generated). The \textbf{normal core} of a subgroup $L$ of $G$ is $\Core_G(L) = \bigcap_G gLg^{-1}$. The following statement is a simple version of Theorem \ref{thm-main-not-loc-fg} below.

\begin{TheoremIntro} \label{thm-intro-main-not-loc-fg}
	Let $\Gamma$ be a finitely generated cocompact lattice of a tdlc group $G$, and $C$ a dense subgroup of $G$ such that $\Gamma \leq C \leq \Comm_G(\Gamma)$. Suppose: \begin{enumerate}
	\item  \label{item-faith-CayAb} $G$ admits a compact open subgroup with trivial normal core;
		\item   \label{item-simplicity-assumption} Every closed normal subgroup of $G$ is  discrete or open; 
		\item \label{item-not-loc-fg} $G$ is not locally finitely generated.
	\end{enumerate} 
	Then every finitely generated commensurated subgroup of $C$ is virtually contained in $\Gamma$. 
\end{TheoremIntro}

Making an assumption on normal subgroups of $G$ is quite natural in this setting. Clearly assumption (\ref{item-simplicity-assumption}) is satisfied if $G$ is topologically simple, but (\ref{item-simplicity-assumption}) is much less stringent than topologically simplicity. Assumption (\ref{item-faith-CayAb}) is very mild, especially in the presence of (\ref{item-simplicity-assumption}). Assumption (\ref{item-not-loc-fg}) is the most restrictive one. It is worth noting Theorem \ref{thm-intro-main-not-loc-fg} fails without this assumption, see Example \ref{ex-intro} below.

When $G \leq \mathrm{Aut}(T)$ is a closed subgroup acting cocompactly on $T$, (\ref{item-faith-CayAb}) is always true, and there are several known  conditions ensuring (\ref{item-simplicity-assumption}) is true, at least when $G$ acts minimally on $T$. One is Tits' independence property \cite{Tits-arbres}, or some of its variations. Another one is that the local action of $G$ is primitive and $G$ has no finite quotient \cite[(1.7)]{BM00a}. Deciding when (\ref{item-not-loc-fg}) is true is delicate, but it is at least true for the entire group $\mathrm{Aut}(T)$. From Theorem \ref{thm-intro-main-not-loc-fg} together with an independent result allowing for the the existence of compact normal subgroups (see Proposition  \ref{prop-latt-cpct-by-discrete} below) we obtain:

\begin{TheoremIntro} \label{thm-intro-main-trees}
	Let $T$ be a locally finite  tree such that  $\Aut(T)$ acts cocompactly on $T$. Let $\Gamma$ be a cocompact lattice of $G = \Aut(T)$. Then, up to commensurability, $\Gamma$ is the only infinite finitely generated commensurated subgroup of $\Comm_G(\Gamma)$.
\end{TheoremIntro}

That theorem in particular yields a positive answer to the aforementioned open problem of Question \ref{quest-intro} for $G = \Aut(T)$ and $\Gamma, \Gamma'$ cocompact: 

\begin{Corollary-intro}
Any two cocompact lattices of $\Aut(T)$ with the same commensurator are commensurable.
\end{Corollary-intro}

Another situation where Theorem \ref{thm-intro-main-not-loc-fg} (or rather Theorem \ref{thm-main-not-loc-fg}) can be applied concerns groups of automorphisms of right angled buildings. Every graph product $\Gamma_P$, associated to a right-angled Coxeter system $(W,S)$ and a family $P = (P_s)_{s \in S}$ of finite groups, naturally sits as a cocompact lattice in the automorphism group of a right-angled building $X$ of type $(W,S)$. We will denote by $\Aut(X)$  the group of all automorphisms of $X$ (i.e.\ automorphisms of the chamber graph of $X$), and by $\Aut_0(X)$ the group of type-preserving automorphisms of $X$. $\Aut_0(X)$ is a closed and cocompact subgroup of $\Aut(X)$, which in general is of infinite index in $\Aut(X)$. This setting yields an extensive family of tdlc groups and cocompact lattices which fall into the general setting we consider. Caprace showed that for $X$ thick, irreducible and of non-spherical type, $\Aut_0(X)$ is abstractly simple  \cite{Cap-RAB} (see also Haglund--Paulin \cite{Hag-Pau-simplicite} for earlier results in the related context of CAT(0) cube complexes, as well as Lazarovich \cite{Lazarovich-cube-complexes-CMH}).  Haglund \cite{Hag-08-GeoDe} and Kubena--Thomas \cite{KubTho} showed the commensurator of $\Gamma_P$ in   $\Aut(X)$ is dense in  $\Aut(X)$. Despite these similarities, this setting also exhibits notable differences with the one of automorphism groups of trees. Beyond the fact that graph products of finite groups form a much richer class of groups than cocompact lattices in automorphism groups of trees (which are all virtually free groups), one remarkable difference is that here the commensurator of $\Gamma_P$ in   $\Aut(X)$ might actually coincides with the abstract commensurator of  $\Gamma_P$. This is the case when the defining graph of $(W,S)$ is a cycle of length $|S| \geq 5$ and $(|P_s|)_{s \in S}$ is constant equal to some integer at least $3$, as a consequence of the Mostow rigidity type theorem proven by Bourdon \cite{Bourdon-Mostow-GAFA97}. Another difference is that not all cocompact lattices in $\Aut(X)$ are conjugate in $\Aut(X)$ up to commensurability. A thorough study of when this holds is carried out by Haglund \cite{Hag-08-GeoDe} and Shepherd \cite{Shepherd-commens-RAB}. In that setting we show:

\begin{TheoremIntro} \label{thm-intro-main-RAB}
	Let $X$ be the semi-regular right angled building associated to a graph product of finite groups $\Gamma_P$, and let $G = \Aut(X)$. Suppose that $X$ is thick and irreducible. Then, up to commensurability, $\Gamma_P$ is the only infinite finitely generated commensurated subgroup of $\Comm_G(\Gamma_P)$. 
\end{TheoremIntro}

As before, the statement implies in particular that any cocompact lattice of $\Aut(X)$ having  the same commensurator as $\Gamma_P$ is commensurable with $\Gamma_P$.

\medskip

As discussed right after the statement, the assumption that the ambient locally compact group $G$ is not locally finitely generated is the most restrictive one in Theorem~\ref{thm-intro-main-not-loc-fg}. We develop a complementary approach that no longer relies on this local condition of $G$. Instead, this second approach is based on local properties of another tdlc group that naturally appears in the present context. Associated to a group $C$ and a commensurated subgroup $\Gamma$, there is a tdlc group $C / \! \! / \Gamma$ and a homomorphism $\tau_{C,\Gamma}: C \to C / \! \! / \Gamma$ with dense image, such that there exists a compact open subgroup $U_\Gamma$ such that $\tau_{C,\Gamma}^{-1}(U_\Gamma) = \Gamma$ and the normal core of $U_\Gamma$ in $C / \! \! / \Gamma$ is trivial. Those properties actually characterize $C / \! \! / \Gamma$ \cite[Lemma 3.6]{Sha-Willis}. The group $C / \! \! / \Gamma$ is called the Schlichting completion of $C$ with respect to $\Gamma$.  If $\mathcal{N}_{C,\Gamma}$ is the collection of finite index normal subgroups $N$ of $\Gamma$ such that $N$ is the intersection of finitely many $C$-conjugates of $\Gamma$, the \textbf{$C$-congruence completion} $\widehat{\Gamma}^C$ of $\Gamma$ is defined as the inverse limit of $\left\lbrace \Gamma / N\right\rbrace $ where $N$ ranges over $\mathcal{N}_{C,\Gamma}$. Then the closure of $\tau_{C,\Gamma}(\Gamma)$ in $C / \! \! / \Gamma$ is a compact open subgroup of $C / \! \! / \Gamma$ isomorphic to $\widehat{\Gamma}^C$.

We recall some terminology. The \textbf{quasi-center} $\QZ(G)$ of a tdlc group $G$ is the set of elements of $G$ whose centralizer is open. The subgroup $\QZ(G)$ is normal in $G$, and $\QZ(G)$ contains every discrete normal subgroup of $G$. The \textbf{discrete residual } $\Res(G)$ of $G$ is the intersection of all open normal subgroups of $G$.   The \textbf{local prime content} of a profinite group $U$ is the set of prime numbers $p$ such that $p$ divides the order of every open subgroup of $U$. Equivalently, $U$ contains an infinite pro-$p$ subgroup.

\begin{TheoremIntro} \label{thm-intro-main-completion}
Let $\Gamma$ be a finitely generated cocompact lattice of a tdlc group $G$, and $C$ a dense subgroup of $G$ such that $\Gamma \leq C \leq \Comm_G(\Gamma)$. Suppose $\QZ(G)$ is discrete, $\Res(G)$ is open, and every abstract normal subgroup of $G$ is contained in $\QZ(G)$ or contains $\Res(G)$. Suppose also the $C$-congruence completion $\widehat{\Gamma}^C$ of $\Gamma$ verifies:
	\begin{enumerate}
		\item \label{item-intro-many-primes} $\widehat{\Gamma}^C$ has infinite local prime content;
		\item \label{item-intro-no-commuting-normal} Any two infinite closed normal subgroups of $\widehat{\Gamma}^C$ have infinite intersection;
		\item \label{item-intro-no-finite-exponent-normal} There is no infinite closed normal subgroup $M$ of $\widehat{\Gamma}^C$ such that $M$ embeds as a closed subgroup of $F^I$ for some finite group $F$ and set $I$.
	\end{enumerate}
	Then every finitely generated commensurated subgroup of $C$ is virtually contained in $\Gamma$.
\end{TheoremIntro}

We prove the theorem under a weaker requirement than (\ref{item-intro-no-commuting-normal}), which consists in asking that condition only for certain normal subgroups. See Theorem \ref{thm-combined-arguments}. 

Observe that the assumption on normal subgroups of $G$ was made on closed normal subgroups in Theorem~\ref{thm-intro-main-not-loc-fg}, while here it concerns abstract normal subgroups (i.e.\ all normal subgroups). The assumptions on normal subgroups here are stronger than those of Theorem~\ref{thm-intro-main-not-loc-fg}, but again much less stringent than simplicity.

Assumption (\ref{item-intro-no-commuting-normal}) is key in the theorem, as the following illustrative  example shows. This example also shows Theorem \ref{thm-intro-main-not-loc-fg} fails without the assumption of local infinite generation of $G$. 

\begin{Example-intro} \label{ex-intro}
  Let $G = \mathbf{H}^1(\Q_p)$ be the group of norm one quaternions over $\Q_p$ for an odd prime $p$. By quaternion we mean the quaternion algebra associated to the parameters $(-1,-1)$. Recall that for an odd prime number the group of norm one quaternions over $\Q_p$ is isomorphic to $\mathrm{SL}(2,\Q_p)$ (which satisfies the requirements of Theorem \ref{thm-intro-main-completion}). The subgroup $\Gamma = \mathbf{H}^1(\Z[1/p])$ is a cocompact lattice of $G$, and $C = \mathbf{H}^1(\Q)$ is a dense subgroup of $G$ that contains and commensurates $\Gamma$. As follows from \cite[Theorem 4.3]{Vigneras-livre-quaternions} and \cite[Lemma 3.6]{Sha-Willis}, the Schlichting completion $C / \! \! / \Gamma$ is isomorphic to the restricted product of the family of groups $\left\lbrace \mathrm{SL}(2,\Q_q) \right\rbrace$ and compact open subgroups $\left\lbrace \mathrm{SL}(2,\Z_q) \right\rbrace$, where $q$ ranges over odd primes different from $p$. So here the $C$-congruence completion $\widehat{\Gamma}^C$ is isomorphic to the direct product $\prod_q \mathrm{SL}(2,\Z_q)$. Hence among the local conditions in Theorem \ref{thm-intro-main-completion}, here (\ref{item-intro-many-primes}) and (\ref{item-intro-no-finite-exponent-normal}) are satisfied, but (\ref{item-intro-no-commuting-normal}) is not. Moreover, the conclusion of the theorem fails here, as $C$ admits infinitely many commensurability classes of  finitely generated commensurated subgroups that are not virtually contained in $\Gamma$, namely $\mathbf{H}^1(\Z[1/\pi])$ for every non-empty finite set of odd primes $\pi$.
\end{Example-intro}

Identifying the profinite group $\widehat{\Gamma}^C$ is a delicate problem. We  say that a commensurated subgroup $\Gamma$ of a group $C$ has the congruence subgroup property (CSP) in $C$ if every finite index subgroup of $\Gamma$ contains an element of $\mathcal{N}_{C,\Gamma}$. Equivalently, the natural surjective homomorphism $\widehat{\Gamma} \to \widehat{\Gamma}^C$ is an isomorphism, where $\widehat{\Gamma}$ is the profinite completion of $\Gamma$. The idea of considering CSP in a setting where $\Gamma$ is a lattice in a group $G$ that is not necessarily algebraic  was suggested by Lubotzky \cite{Mozes-CSP-tree}. When $\Gamma$ is an arithmetic lattice in a simple algebraic group and $C = \Comm_G(\Gamma)$, this notion coincides with the classical one \cite[Proposition 1.1]{Mozes-CSP-tree}. The interest of this notion in the context of Theorem \ref{thm-intro-main-completion} is that in certain situations, enough is known on $\Gamma$ (as an abstract group, independently on the way $\Gamma$ sits inside $G$) to ensure that its profinite completion $\widehat{\Gamma}$ satisfies the conditions (\ref{item-intro-many-primes}), (\ref{item-intro-no-commuting-normal}), (\ref{item-intro-no-finite-exponent-normal}) from  Theorem \ref{thm-intro-main-completion}. This is for instance the case if $G$ is a closed and cocompact subgroup of the automorphism group of an infinitely ended tree. In that situation $\Gamma$ admits a finite index subgroup that is a non-abelian free group, and $\widehat{\Gamma}$ admits a finite index subgroup that is a non-abelian free profinite group. This guarantees  (\ref{item-intro-many-primes}), (\ref{item-intro-no-commuting-normal}), (\ref{item-intro-no-finite-exponent-normal}) hold. More generally, if $\Gamma$ is Gromov-hyperbolic and virtually special, then $\widehat{\Gamma}$ satisfies the conditions needed in Theorem \ref{thm-combined-arguments} (see Proposition \ref{prop-completion-hyp-special}). So for $\Gamma, C, G$ as in Theorem \ref{thm-intro-main-completion} and for $\Gamma$ hyperbolic and virtually special, we therefore obtain a criterion (conditionally to the fact that $\Gamma$ has CSP in $C$) ensuring every finitely generated commensurated subgroup of $C$ is virtually contained in $\Gamma$ (see Corollary \ref{cor-hyp-special-CSP}). 

We effectively manage to implement this criterion for certain closed subgroups of automorphism groups of trees. Mozes showed that when $G = \mathrm{Aut}(T_d)$ is the full automorphism group of a regular tree $T_d$ and $\Gamma$ is a cocompact lattice in $G$, then $\Gamma$ has the CSP in $\Comm_G(\Gamma)$ \cite[Theorem 1.2]{Mozes-CSP-tree}. Mozes' argument can be generalized to show that the same holds true within the subgroup $U(F)$ of $\mathrm{Aut}(T_d)$ consisting of automorphisms that have local action prescribed by the finite permutation group $F$; see Theorem \ref{thm-CSP-U(F)} below. For this family of groups we prove:

\begin{TheoremIntro} \label{thm-intro-U(F)}
Let $d \geq 3$, $F \leq \mathrm{Sym}(d)$ and $\Gamma$ a cocompact lattice of $G = U(F)$. Then, up to commensurability, $\Gamma$ is the only infinite finitely generated commensurated subgroup of $\Comm_G(\Gamma)$.
\end{TheoremIntro}

\noindent 
\textbf{Organization.} Section \ref{sec-prelim} provides some preliminary material. Section \ref{sec-join} establishes a key tool for the rest of the paper. The main result there, which might potentially be of interest in other contexts involving commensurated subgroups, defines a join operation on (commensurability classes of) finitely generated commensurated subgroups of a given group. See Theorem~\ref{thm:commensurated_join} and Corollary \ref{cor:commensurated_join}. The proofs of the two general results stated in this introduction, Theorem \ref{thm-intro-main-not-loc-fg} and  Theorem \ref{thm-intro-main-completion}, are given respectively in Section \ref{sec-not-loc-fg} and Section \ref{sec-local-assumption-completion}. The applications to automorphism groups of trees and buildings are carried out respectively in Section \ref{sec-trees} and Section \ref{sec-buildings}. 

\bigskip

\noindent 
\textbf{Notation.} Everywhere in the paper, whenever $C$ is a group and $H$ is a subgroup of $C$, we write $[H]$ for the commensurability class of $H$ in $C$, i.e.\ the collection of all subgroups of $C$ commensurable with $H$. This notation does not make $C$ appear, but the group $C$ will always be clear from the context.

\bigskip

\noindent 
\textbf{Acknowledgments.} We thank Pierre-Emmanuel Caprace for interesting comments on this work.

\section{Preliminaries} \label{sec-prelim}

\subsection{On the discrete residual}

The relevance of the discrete residual $\Res(G)$ of a tdlc group $G$ will appear repeatedly in the paper, notably (but not only) through the following result, the main contribution of which is the equivalence between (\ref{item-res-triv}) and (\ref{item-SIN}) due to Caprace--Monod. 

\begin{Theorem} \label{thm-res-disc-res-triv}
	For $G$ a compactly generated tdlc group, the following are equivalent: \begin{enumerate}
		\item \label{item-res-disc} $\Res(G)$ is a discrete subgroup of $G$; \item \label{item-res-triv} $\Res(G)$ is trivial; \item \label{item-SIN}  The compact open normal subgroups form a basis of neighborhoods of $1$ in $G$. 
	\end{enumerate}
\end{Theorem}

\begin{proof}
See \cite[Theorem G]{Reid-distal} or \cite[Proposition 2.2]{CLB-covol} for (\ref{item-res-disc}) $\Leftrightarrow$ (\ref{item-res-triv}), and \cite[Corollary~4.1]{CaMo-decomp} for  (\ref{item-res-triv}) $\Leftrightarrow$ (\ref{item-SIN}). 
\end{proof}

Although very simple, the following observation will be particularly useful. 

\begin{Lemma} \label{lem-commens-normalizes-res}
	Let $G$ be a tdlc group, and $L$ a closed subgroup of $G$. Then: \begin{enumerate}
		\item If $L_1$ is a finite index open subgroup of $L$, then $\Res(L_1) = \Res(L)$.
		\item The commensurator of $L$ in $G$ normalizes $\Res(L)$. 
	\end{enumerate} 
\end{Lemma}

\begin{proof}
The inclusion $\Res(L_1) \leq \Res(L)$ follows from the definition. For the converse, let $O$ be an open normal subgroup of $L_1$. Then $O$ is open in $L$, and has finitely many $L$-conjugates. Their intersection $O'$ is an open normal subgroup of $L$, so $\Res(L) \leq O' \leq O$. Hence $\Res(L) \leq \Res(L_1)$. 

For $g$ in the commensurator of $L$ in $G$, we have $\Res(g^{-1} L g \cap L) = \Res(L \cap g L g^{-1}) = \Res(L)$ by the first point. Hence conjugation by $g$ stabilizes $\Res(L)$. 
\end{proof}

\subsection{On compact normal subgroups}

The following follows from \cite[Theorem 5.5]{Wang-compactness}. 

\begin{Proposition}\label{prop-coc-cpt-normal}
	Let $G$ be a locally compact group and $H$ a closed cocompact subgroup. Then every compact normal subgroup of $H$ is contained in a compact normal subgroup of $G$. 
\end{Proposition}

At several places of the paper we will discuss the property that a tdlc group admits a compact open subgroup with trivial normal core. The following lemmas collect basic observations around this.

\begin{Lemma} \label{lem-faith-Cay-Ab-coco}
	Let $G$ be a tdlc group and $L$ a closed  cocompact subgroup. If $G$ has a compact open subgroup with trivial normal core, then so does $L$. 
\end{Lemma}

\begin{proof}
Let $U$ compact open such that $\Core_G(U)$ is trivial. Take representatives $g_1 U$, $\ldots$, $g_n U$ of the $L$-orbits in $G/U$, and let $V = \bigcap_i g_i U g_i^{-1}$. Then $V_L := L \cap V$ is compact open in $L$ and $\Core_L(V_L)$ is trivial. 
\end{proof}

\begin{Lemma} \label{lem-faith-Cay-Ab-variety}
	Let $G$ be a tdlc group and $U \leq G$ a compact open subgroup with trivial normal core. Then every compact normal subgroup of $G$ embeds as a closed subgroup of $F^I$ for some finite group $F$ and set $I$.
\end{Lemma}

\begin{proof}
Let $K$ be a compact normal subgroup of $G$. The subgroup $K$ acts with finite orbits on $G/U$, and since $K$ is normal the group $G$ preserves and acts transitively on the $K$-orbits. We just take $F$ to be the finite permutation group induced by $K$ on each orbit. 
\end{proof}

\begin{Proposition} \label{prop-latt-fi-normalizer}
	Let $\Gamma$ be a finitely generated cocompact lattice of a tdlc group $G$, and suppose $G$ admits a compact open subgroup with trivial normal core. Then $\Gamma$ has finite index in its normalizer $N_G(\Gamma)$. 
\end{Proposition}

\begin{proof}
Since $\Gamma$ is discrete and normal in $N_G(\Gamma)$, $\Gamma$ lies in the quasi-center of $N_G(\Gamma)$. Since $\Gamma$ is finitely generated, one can find a compact open subgroup $U$ of $N_G(\Gamma)$ such that $U$ centralizes $\Gamma$. For every open normal subgroup $V$ of $U$, the normalizer of $V$ in $N_G(\Gamma)$ therefore contains $U$ and $\Gamma$, and hence has finite index in $N_G(\Gamma)$ because $\Gamma$ is cocompact in $N_G(\Gamma)$. On the other hand $N_G(\Gamma)$ does admit a compact open subgroup $V$ with trivial normal core (Lemma  \ref{lem-faith-Cay-Ab-coco}). So we infer that $N_G(\Gamma)$ must be discrete, and $\Gamma$ has finite index in $N_G(\Gamma)$.
\end{proof}

\begin{Lemma} \label{lem-simplicity-assumptions-faith-Cay-Ab}
	Let $G$ be a compactly generated tdlc group such that $\res(G)$ is open in $G$, and every closed normal subgroup of $G$ is either discrete or contains $\res(G)$. Then $G$ admits a compact open subgroup with trivial normal core.		
\end{Lemma}

\begin{proof}
If $\res(G)$ is trivial then $G$ is discrete and the conclusion trivially holds. Otherwise choose a compact open subgroup $U$ of $G$ properly contained in $\res(G)$. Then $\Core_G(U)$ must be discrete by the assumption. Hence $\Core_G(U)$ is finite, and any open subgroup $V$ of $U$ that intersects $\Core_G(U)$ trivially has trivial normal core in $G$. 
\end{proof}

\subsection{Completions}

Let $\Gamma$ be a commensurated subgroup of a group $C$. The action of $C$ on $\Omega_\Gamma := C / \Gamma$ yields a homomorphism $\tau_{C,\Gamma}: C \to \mathrm{Sym}(\Omega_\Gamma)$. The {Schlichting completion} $C / \! \! / \Gamma$ of $C$ with respect to $\Gamma$ is the closure of $\tau_{C,\Gamma}(C)$. $\Gamma$ acts on $\Omega_\Gamma$ with finite orbits, and the closure of $\tau_{C,\Gamma}(\Gamma)$ in $C / \! \! / \Gamma$ is isomorphic to the $C$-congruence completion  $\widehat{\Gamma}^C$. 

The {completion of $C$ with respect to the commensurability class} $[\Gamma]$ is defined similarly as $C / \! \! / \Gamma$, replacing $\Omega_\Gamma$ by the disjoint union $\Omega = \bigsqcup \Omega_{\Gamma'}$ where $\Gamma'$ ranges over $[\Gamma]$. It is denoted $ C / \! \! / [\Gamma]$, and the homomorphism from $C$ to $ C / \! \! / [\Gamma]$ is denoted $\tau_{C,[\Gamma]}: C \to C / \! \! / [\Gamma]$. The closure of $\tau_{C,[\Gamma]}(\Gamma)$ in $C / \! \! / [\Gamma]$ is isomorphic to the profinite completion  $\widehat{\Gamma}$. By construction there is a surjective homomorphism $C / \! \! / [\Gamma] \to C / \! \! / \Gamma$, whose restriction to $\widehat{\Gamma}$ is the natural homomorphism $\widehat{\Gamma} \to \widehat{\Gamma}^C$ (where we identify $\widehat{\Gamma}$ and $\widehat{\Gamma}^C$ with the image closure of $\Gamma$ in $C / \! \! / [\Gamma]$ and $C / \! \! / \Gamma$ respectively). The kernel of $C / \! \! / [\Gamma] \to C / \! \! / \Gamma$ is the normal core of $\widehat{\Gamma}$ in $C / \! \! / [\Gamma]$. The subgroup $\Gamma$ has the CSP in $C$ if $C / \! \! / [\Gamma] \to  C / \! \! / \Gamma$ is an isomorphism. 

We have the following basic lemma.

\begin{Lemma} \label{lem-CSP-one-all}
	Let $\Gamma$ be a commensurated subgroup of a group $C$. Suppose: \begin{enumerate}
		\item $\Gamma$ has the CSP in $C$; 
		\item Every element of $[\Gamma]$ is residually finite and virtually torsion-free;
		\item Every torsion-free element in $[\Gamma]$ has torsion-free profinite completion.
	\end{enumerate}
	Then every $\Gamma' \in [\Gamma]$ with no non-trivial finite normal subgroup has the CSP in $C$.
\end{Lemma}

\begin{proof}
	Since CSP passes from $\Gamma$ to a finite index subgroup, $\Gamma \cap \Gamma'$ has the CSP in $C$. So we can assume $\Gamma \leq \Gamma'$. Also we can assume $\Gamma$ is torsion-free. Let $N$ be the kernel of $ C / \! \! / [\Gamma] \to  C / \! \! / \Gamma'$. The assumption that $\Gamma$ has the CSP in $C$ means that $\widehat{\Gamma}$ has trivial normal core in $ C / \! \! / [\Gamma]$. So by Lemma  \ref{lem-faith-Cay-Ab-variety} the group $N$ is torsion. Since $\Gamma$ is torsion-free, the assumptions imply $\widehat{\Gamma}$ is torsion-free. So $N \cap \widehat{\Gamma}$ is trivial, $N$ is finite and $\langle N,  \widehat{\Gamma}\rangle  \simeq N \times \widehat{\Gamma}$. Let $\Gamma^{''}$ be the pre-image in $\Gamma'$ of $N \times \widehat{\Gamma}$. We have $\Gamma \leq \Gamma^{''} \leq \Gamma^{'}$, and $\widehat{\Gamma^{''}} \simeq N \times \widehat{\Gamma}$. If $N$ is non-trivial then $\widehat{\Gamma^{''}}$ has torsion, and therefore $\Gamma^{''}$ also has torsion by the assumptions. Since $\Gamma^{''}$ embeds in $\widehat{\Gamma^{''}}$, and the only torsion elements of $N \times \widehat{\Gamma}$ are the elements of $N$, we deduce $\Gamma^{''} \cap N$ is non-trivial. In particular $\Gamma^{'} \cap N$ is a non-trivial finite normal subgroup of $\Gamma^{'}$, a contradiction.  
\end{proof}

\section{A join operation for commensurated subgroups} \label{sec-join}

The goal of this section is to prove Theorem~\ref{thm:commensurated_join}. As a consequence of it, we derive that in any group, in the poset of commensurability classes of subgroups, any pair of finitely generated commensurated subgroups has a unique least upper bound; which is also a finitely generated commensurated subgroup (Corollary \ref{cor:commensurated_join}).

\begin{Definition}
Let $G$ be a tdlc group and $H$ a subgroup of $G$.  An \textbf{open envelope} for $H$ in $G$ is an open subgroup $E$ of $G$ such that $H \leq E$, and for every open subgroup $O$ of $G$ such that $H$ is virtually contained in $O$, $E$ is virtually contained in $O$.
\end{Definition}

It is easy to see that if $H_1$ and $H_2$ are commensurable subgroups of $G$ with open envelopes $E_1$ and $E_2$ in $G$, then $E_1$ and $E_2$ are also commensurable.  In particular, an open envelope, if it exists, is  unique up to commensurability.  Note also that $E$ is an open envelope for $H$ if and only if $E$ is an open envelope for the closure $\overline{H}$: this follows from observing that for any open subgroup $O$, we have $H$  virtually contained in $O$ if and only if $\overline{H}$ is.

The following result, which will be used as a tool in the proof of Theorem~\ref{thm:commensurated_join}, ensures that when $H$ is generated by a relatively compact subset of $G$, there is always an open envelope for $H$ in $G$.  This result was proved in \cite{Reid-distal}. The proof given there is rather long, and relies on an auxiliary result from \cite{Reid-orbit-closures-GGD}.  We provide a short and self-contained proof. Write $\mc{O}(G,H)$ for the set of open subgroups of $G$ normalized by $H$ and $\Res_G(H) = \bigcap \mc{O}(G,H)$.

\begin{Theorem}\label{thm:open_envelope}
Let $G$ be a tdlc group, and let $S$ be a relatively compact subset of $G$.  Then $H = \langle S \rangle$ has an open envelope in $G$.  Moreover, for any open envelope $E$ of $H$, we have $\Res_G(H) = \Res(E)$.
\end{Theorem}

We say a family of sets $\mc{S}$ is {filtering} if for all $A,B \in \mc{S}$, there exists $C \in \mc{S}$ such that $C \subseteq A \cap B$.  The proof of the following proposition is inspired by \cite[Proposition 2.5]{CaMo-decomp}. 

\begin{Proposition}\label{prop:small_invariant_neighborhood}
Let $X$ be a locally compact space, and $H$ a compactly generated locally compact group acting continuously on $X$. Let $x \in X$ and $\mc{O}$ be a filtering family of $H$-invariant closed sets containing $x$ such that $\bigcap \mc{O} = \{x\}$.  Then for every compact open neighborhood $W$ of $x$ in $X$, there exists $O \in \mc{O}$ such that $W \cap O$ is $H$-invariant.
\end{Proposition}

\begin{proof}
Take $S$ compact in $H$ such that  $S = S^{-1}$ and $H = \langle S \rangle$. Since $W$ is compact and open, the stabilizer of $W$ in $H$ is open \cite[III,5,Th.1]{Bourbaki-top-gen}. Therefore $\{ sW \mid s \in S \}$ is finite, and $W_0 := \bigcap_{s \in S}sW$ is open. Note that since $S = S^{-1}$ we have $sW_0 \subseteq W$ for every $s \in S$. Let $W_1 = W \setminus W_0$.  We note that $H$ fixes $x$, so $x \not\in W_1$ and hence $W_1 \cap \bigcap \mc{O} = \emptyset$.  Since $W_1$ is compact and $\mc{O}$ is filtering, it follows that there is $O \in \mc{O}$ such that $W_1 \cap O = \emptyset$. Hence  $W \cap O = W_0 \cap O$. By the construction of $W_0$, it follows that $s (W \cap O) \subseteq W \cap O$ for every $s \in S$.  Since $S = S^{-1}$ and $H = \langle S \rangle$, $W \cap O$ is $H$-invariant.
\end{proof}

\begin{proof}[Proof of Theorem~\ref{thm:open_envelope}]
Write $R = \Res_G(H)$.	Consider the closure $\overline{H}$ of $H$ in $G$. It is generated by $\overline{H}$ together with any neighborhood of $1$ in $\overline{H}$. Hence it is compactly generated. Since $\Res_G(H) = \Res_G(\overline{H})$, we can assume $H$ is closed and compactly generated. Let $\mc{O} = \{O/R \mid O \in \mc{O}(G,H)\}$, considered as a collection of closed subspaces of $X = G/R$.  The elements of $\mc{O}$ are invariant under the action of $H$ on $X$ by conjugation, $\mc{O}$ is closed under finite intersections, and $\bigcap\mc{O} = \{R\}$. Applying Proposition~\ref{prop:small_invariant_neighborhood} and lifting up to $G$, we deduce that for every compact open subgroup $U$ of $G$ there is $O \in \mc{O}(G,H)$ such that $O \cap UR$ is $H$-invariant. In particular there is a compact open subgroup $V$ of $G$ all of whose $H$-conjugates lie in $UR$. We check $V$ normalizes $R$. Take $v \in V$ and $O \in \mc{O}(G,H)$. It follows from the property satisfied by $V$ that $vOv^{-1}$ has only finitely many $H$-conjugates. Hence their intersection belongs to $\mc{O}(G,H)$, and hence $R \leq vOv^{-1}$. Since $O$ was arbitrary in $\mc{O}(G,H)$ we have $R \leq vRv^{-1}$, and then $R = vRv^{-1}$. Hence $V$ normalizes $R$. Therefore $VR$ is a subgroup of $G$, and hence if $O \in \mc{O}(G,H)$ is such that $O' = O \cap VR$ is $H$-invariant, then $O' \in \mc{O}(G,H)$ and $O'$ contains $R$ as a cocompact normal subgroup. It is then a simple verification to see that $E = O' H$  is an open envelope for $H$. 

It remains to check that $R = \Res(E)$ (which by Lemma \ref{lem-commens-normalizes-res} is equivalent to $R = \Res(E')$ for any open envelope $E' $ for $H$). The inclusion $R \leq \Res(E)$ is clear. Conversely, for $O \in \mc{O}(E,H)$ we have that the normalizer of $O$ in $E$ has finite index by the defining property of an open envelope. Therefore the intersection of all $E$-conjugates of $O$ is an open normal subgroup of $E$ contained in $O$, yielding the inclusion  $\Res(E) \leq R$.
\end{proof}

\begin{Remark} \label{rmk-envelope-cpct-gen}
If $E$ is an open envelope of $H$, then it follows from the definition that $\left\langle H, U \right\rangle $ has finite index in $H$ for every compact open subgroup $U$ of $E$. In particular for $H$ as in Theorem~\ref{thm:open_envelope}, any open envelope of $H$ is compactly generated. 
\end{Remark}

Let $C$ be a group with commensurated subgroups $A_1,A_2$.  In general, the subgroup $\langle A_1,A_2 \rangle$ is not commensurated by $C$, and its commensurability class is sensitive to the choice of $A_1$ and $A_2$ in $[A_1]$ and $[A_2]$. A basic example is when $C = A_1 *_B A_2$, where the amalgamating subgroup $B$ maps to both a proper finite index subgroup of $A_1$ and a proper finite index subgroup of $A_2$.  However, one might ask whether, among subgroups in which $A_1,A_2$ are virtually contained, there is one that is the unique smallest such up to commensurability. Provided it exists, given its uniqueness such a subgroup is necessarily commensurated in $C$.  The following result gives a positive answer in certain circumstances. 

\begin{Theorem}\label{thm:commensurated_join}
	Let $C$ be a group with subgroups $A_1,A_2$ satisfying one of the following assumptions:
	\begin{enumerate}[label={(\alph*)}]
		\item \label{A1-fg} $A_2$ is commensurated in $C$, and $A_1$ is generated by $A_1 \cap A_2$ together with finitely many elements;
		\item \label{A2-fg} both $A_1$ and $A_2$ are commensurated in $C$, and $A_2$ is generated by $A_1 \cap A_2$ together with finitely many elements.
	\end{enumerate}
	Then there is a subgroup $A$ of $C$, unique up to commensurability, such that:
	\begin{enumerate}
		\item \label{item-A-generators} $A_1$ and $A_2$ are virtually contained in $A$; 
		\item \label{item-A-smallest} For every subgroup $B$ of $C$ such that $A_1$ and $A_2$ are virtually contained in $B$, $A$ is virtually contained in $B$.
	\end{enumerate}
	Moreover one can take $A = \langle A_1,A_2' \rangle$ for some finite index subgroup $A_2'$ of $A_2$, and $A_2'$ can be taken to be a finite intersection of conjugates of $A_2$;
\end{Theorem}

\begin{proof}
	Observe that uniqueness of $A$ up to commensurability is automatic from (\ref{item-A-generators}) and (\ref{item-A-smallest}). Let us first prove the result under hypothesis \ref{A1-fg}. Consider the completion $C / \! \! / A_2$ of $C$ with respect to $A_2$. For simplicity we write $L = C / \! \! / A_2$ and $\tau$ instead of $\tau_{C,A_2}$ for the homomorphism  $C \to L$. 

Since  $\tau(A_2)$ is relatively compact in $L$, \ref{A1-fg} ensures $\tau(A_1)$ is generated by a relatively compact subset of $L$. Hence we can appeal to Theorem~\ref{thm:open_envelope} to obtain an open envelope $E$ of $\tau(A_1)$ in $L$. Let $A_0 = \tau^{-1}(E)$.  Then $A_0$ contains $A_1$, and $A_0$ contains a finite intersection $A_2'$ of conjugates of $A_2$ because $E$ is open in $L$. In particular $A_2$ is virtually contained in $A_0$. Now suppose $B$ is a subgroup of $C$ such that $A_1$ and $A_2$ are virtually contained in $B$. Consider $O = \overline{\tau(B)}$. Since  $A_2$ is virtually contained in $B$, $O$ is open in $L$. Since $\tau(A_1)$ is virtually contained in $O$, by the characterization of $E$ we infer that $E$ is virtually contained in $O$. Hence $\tau^{-1}(E) = A_0$ is virtually contained in  $ \tau^{-1}(O)$. We now claim that $B$ has finite index in $\tau^{-1}(O)$. Indeed, letting $U = \overline{\tau(A_2 \cap B)}$, we see that $\tau^{-1}(U)$ is commensurable with $A_2$, and hence $\tau^{-1}(U)$ contains $A_2 \cap B$ as a finite index subgroup. Combined with the equation $O = \tau(B) U$, we infer that $\tau^{-1}(O) = B \tau^{-1}(U)$ contains $B$ with finite index. So we deduce $A_0$ is virtually contained in $B$. Taking $A = \langle  A_1,A_2' \rangle$ satisfies the conclusions of the theorem: note that $A \le A_0$ by construction, so $A$ inherits (\ref{item-A-smallest}) from $A_0$.

Now suppose instead that hypothesis \ref{A2-fg} holds.  By the previous argument we obtain a subgroup ${A}'$ with property (\ref{item-A-smallest}) of the form ${A}' = \langle A_1', A_2 \rangle$  with $A_1'$  a finite index subgroup of $A_1$. The point is to see that there is a subgroup $A$ commensurable with ${A}'$ with $A$ as in the last sentence of the statement. Since  $A_1$ and $A_2$ are commensurated, any conjugate of ${A}'$ keeps (\ref{item-A-generators}) and (\ref{item-A-smallest}), and hence is commensurable with $A'$. Hence ${A}'$ is commensurated in $C$. In particular any two $A_1$-conjugates of ${A}'$ are commensurable. Now since $A_1$ is virtually contained in ${A}'$, there are only finitely many of those conjugates. Hence their intersection provides a finite index subgroup $A''$ of $A'$ that is normalized by $A_1$, and $A''' = A'' A_1$ remains commensurable with $A'$. By construction $A'''$ contains a subgroup $A_2'$ of $A_2$ that is a finite intersection of conjugates of $A_2$ (as $A''$ already does). It follows from (\ref{item-A-smallest}) that the subgroup $A = \langle A_1,A_2' \rangle$ of $A'''$ satisfies the conclusion. 
\end{proof}

Given a group $C$, we denote by $\mathrm{S}(C)$ the set of subgroups of $C$, and $\mathrm{S}(C) / \! \sim$ the set of commensurability classes of subgroups of $C$. Write $[A_1] \preccurlyeq [A_2]$ if $A_1$ is virtually contained in $A_2$. This relation makes $\mathrm{S}(C) / \! \sim$  a poset. 

\begin{Corollary}\label{cor:commensurated_join}
	Let $C$ be a group with commensurated subgroups $A_1$ and $A_2$, at least one of which is finitely generated.  Then $\left\lbrace [A_1] , [A_2] \right\rbrace $ has a least upper bound $[A_1] \vee [A_2]$ in $\mathrm{S}(C) / \! \sim$, and $[A_1] \vee [A_2]$ consists of commensurated subgroups. Moreover we can find $A$ such that $[A] = [A_1] \vee [A_2]$ of the form $A = \langle A_1, A_2 \cap A\rangle$.  In particular, if $A_1$ and $A_2$ are both finitely generated, then so is $A$. We call $[A]$ the \textbf{join} of $[A_1]$ and $[A_2]$. 
\end{Corollary}

\begin{proof}
Apply case \ref{A1-fg} of Theorem~\ref{thm:commensurated_join}. Since both $A_1,A_2$ are commensurated, the subgroup $A$ provided by the theorem satisfies the present conclusion. 
\end{proof}

\section{The proof of Theorem \ref{thm-intro-main-not-loc-fg}} \label{sec-not-loc-fg}

The main goal of this section is to prove Theorem \ref{thm-main-not-loc-fg}, of which Theorem \ref{thm-intro-main-not-loc-fg} from the introduction is a special case.

\begin{Proposition} \label{prop-one-latt-normalizes-other}
	Let $\Gamma$ be a finitely generated cocompact lattice of a tdlc group $G$, and suppose $G$ admits a compact open subgroup with trivial normal core. Let $\rho: \Comm_G(\Gamma) \to \Comm(\Gamma)$ be the natural homomorphism from the commensurator of $\Gamma$ in $G$ to the abstract commensurator of $\Gamma$. If $H$ is a finitely generated subgroup of $\Comm_G(\Gamma)$ such that  $\rho(H)$ is virtually contained in $\rho(\Gamma)$, then $H$ is virtually contained in $\Gamma$. 
\end{Proposition}

\begin{proof}
Upon passing to a finite index subgroup of $H$ one can assume $\rho(H) \leq \rho(\Gamma)$. So every $h$ in $H$ can be written $c \gamma$ with $c \in G$ centralizing a finite index subgroup of $\Gamma$, and $\gamma \in \Gamma$. Take $h_1,\ldots,h_n$ a finite generating subset of $H$, and write $c_i \gamma_i = h_i$ as above. Let $\Gamma_i$ be a finite index subgroup of $\Gamma$ centralized by $c_i$. We can assume $\Gamma_i$ is normal in $\Gamma$. By construction $\Gamma' = \bigcap^n_{i=1}\Gamma_i$ is normalized by $H$. By Proposition \ref{prop-latt-fi-normalizer}, $H$ must be virtually contained in $\Gamma'$, and hence in $\Gamma$.
\end{proof}

\begin{Proposition} \label{prop-latt-cpct-by-discrete}
		Let $\Gamma$ be a finitely generated cocompact lattice of a tdlc group $G$, and suppose $G$ admits a compact open subgroup with trivial normal core. Suppose $K$ is a compact normal subgroup of $G$ and $H$ a finitely generated subgroup of $\Comm_G(\Gamma)$ such that $KH$ is virtually contained in $K \Gamma$. Then $H$ is virtually contained in $\Gamma$.
\end{Proposition}

\begin{proof}
	Upon passing from $H$ to a finite index subgroup one can assume $H \leq K \Gamma$. Hence upon replacing $G$ by $K \Gamma$ (which also admits a compact open subgroup with trivial normal core by Lemma  \ref{lem-faith-Cay-Ab-coco}), one can assume $G = K \Gamma$. In particular $\Comm_G(\Gamma) = (K \cap \Comm_G(\Gamma)) \Gamma$. Also, upon replacing $G$ by the closure of $\Comm_G(\Gamma)$ one can assume $\Comm_G(\Gamma)$ is dense. 
	
	Set $F := K \cap \Gamma$. Since $\Gamma$ is discrete and $K$ is compact, $F$ is finite. We want to reduce to the case $F$ is trivial. Consider the intersection $N$ of all finite index subgroups of $\Gamma$, and let $E = F \cap N = K \cap N$. The subgroup $N$ is normalized by $\Comm_G(\Gamma)$. Therefore so is $E$ (because $K$ is normal in $G$). Therefore the finite subgroup $E$ has dense normalizer in $G$, and hence is normal in $G$. Hence upon considering $G/E$ one can assume $E$ is trivial. That means there is a finite index subgroup of $\Gamma$  that intersects $K$ trivially, and hence there is no loss of generality in assuming $K \cap \Gamma$ is trivial.
	
	Now take $c \in K \cap \Comm_G(\Gamma)$, and let $\gamma \in \Gamma \cap c^{-1} \Gamma c$. The commutator $[c,\gamma]$ belongs to $\Gamma$ by definition, and also to $K$ because $c \in K $ and $K$ is normal. Therefore $[c,\gamma]$ is trivial, and $c$ centralizes $\Gamma \cap c^{-1} \Gamma c$. This means $K \cap \Comm_G(\Gamma)$ lies in the kernel of $\rho: \Comm_G(\Gamma) \to \Comm(\Gamma)$, and in view of the equation $\Comm_G(\Gamma) = (K \cap \Comm_G(\Gamma)) \Gamma$ it follows that the image of $\rho$ is equal to $\rho(\Gamma)$. The statement then follows from Proposition \ref{prop-one-latt-normalizes-other}. 
\end{proof}

We now return to our main setting where $\Gamma$ is a cocompact lattice of $G$, and $C$  a dense subgroup of $G$ such that $\Gamma \leq C \leq \Comm_G(\Gamma)$. One important feature of the completion $C/\! \!/\Gamma$ in that situation is that $C$ embeds as an irreducible cocompact lattice in the product $G \times C/\! \!/\Gamma$ \cite{CM-NPC-disc-sub}. The following proposition recasts properties of the join operation defined in Corollary \ref{cor:commensurated_join} in that situation.

\begin{Proposition} \label{prop-join-lattice-situation}
	Let $G$ be a locally compact group, and $\Gamma \leq G$ a cocompact lattice. Let $C$ be a dense subgroup of $G$ such that $\Gamma \leq C \leq \Comm_G(\Gamma)$, and let $H$ be a finitely generated  commensurated subgroup of $C$. Let $\tau_{C,\Gamma} : C \to C/\! \!/\Gamma$ be the homomorphism from $C$ to $C/\! \!/\Gamma$. Let $E$ be an open envelope of $\tau_{C,\Gamma}(H)$ in  $C/\! \!/\Gamma$, let $\Lambda = \tau_{C,\Gamma}^{-1}(E)$ (so that $\Lambda$ is a representative of the join of $[\Gamma]$ and $[H]$), and let $L = \overline{\Lambda} \leq G$. Then: \begin{enumerate}
	\item the diagonal homomorphism $\Lambda \to L \times E$ has discrete and cocompact image, and the projection of $\Lambda$ to each factor is dense. 
\end{enumerate}
Moreover for $G$ tdlc, we have  :  \begin{enumerate}[resume]
	\item the group $L$ is locally finitely generated; 
	\item $\Res(L)$ is a closed normal subgroup of $G$. 
\end{enumerate}
\end{Proposition}

\begin{proof}
We note that $E$ indeed exists since $H$ is finitely generated. Write $G_1 = G$ and $G_2 = C/\! \!/\Gamma$. Since $\Gamma$ is discrete and cocompact in $G_1$, and $C$ maps densely to $G_2$, and there is a compact open subgroup of $G_2$ whose pre-image in $C$ is equal to $\Gamma$, the image of $C$ in $G_1 \times G_2$ under the diagonal homomorphism is discrete and cocompact (see \cite[Lemma 5.15]{CM-NPC-disc-sub}).  For the rest of the proof we identify $C$ and its subgroups with their images in $G_1 \times G_2$. By definition one has $\Lambda = C \cap (G_1 \times E)$, so that $\Lambda$ is indeed discrete cocompact in $G_1 \times E$. By definition of $L$ we have that $\Lambda$ sits inside $L \times E$, and has a dense projection on each of $L$ and $E$. 

If $G$ is tdlc and $U$ is a compact open subgroup of $L$, then $\Lambda_U = \Lambda \cap (U \times E)$ is discrete and cocompact in $U \times E$. The subgroup $H$ being finitely generated, $E$ is compactly generated (Remark \ref{rmk-envelope-cpct-gen}). Therefore so is $U \times E$, and we infer $\Lambda_U$ is finitely generated (being a discrete and cocompact subgroup of a compactly generated group). By assumption $\Lambda_U$ has a dense projection to $U$, so $U$ is topologically finitely generated. (This observation goes back at least to  \cite[Proposition 1.1.2]{Burger-Mozes-Zimmer--linear-rep}). 

Finally the subgroup $\Lambda$ is commensurated in $C$ by Corollary \ref{cor:commensurated_join}. Therefore its closure $L$ in $G$ remains commensurated by $C$ (as is easily verified, see for instance \cite[Lemma 2.7]{LB-Wes-commens}). By Lemma  \ref{lem-commens-normalizes-res} this implies $C$ normalizes $\Res(L)$. Since  $C$ is dense in $G$ and $\Res(L)$ is closed, $\Res(L)$ is normal in $G$. 
\end{proof}

Theorem \ref{thm-intro-main-not-loc-fg} from the introduction is a special case of the following result. 

\begin{Theorem} \label{thm-main-not-loc-fg}
Let $\Gamma$ be a finitely generated cocompact lattice of a tdlc group $G$, and $C$ a dense subgroup of $G$ such that $\Gamma \leq C \leq \Comm_G(\Gamma)$. Suppose: \begin{enumerate}
	\item  \label{item-faith-CayAb} $G$ admits a compact open subgroup with trivial normal core;
	\item \label{item-alternative-normal-subgroups} closed normal subgroups $N$ of $G$ satisfy the following alternative: either there exists a compact subgroup $K$ of $N$ such that $K$ is open in $N$ and $K$ is normal in $G$, or $\Res(G) \leq N$; 
	\item \label{item-not-loc-fg} for every closed subgroup $J$ of $G$ satisfying $\Res(G) \leq J$, the group $J$ is not locally finitely generated.
\end{enumerate} 
Then every finitely generated commensurated subgroup of $C$ is virtually contained in $\Gamma$. 
\end{Theorem}

\begin{proof}
Take $H \leq C$ finitely generated and commensurated, and let $\Lambda$ and $L$ be as in Proposition  \ref{prop-join-lattice-situation}. After replacing $\Gamma$ with a finite index subgroup, we may assume $\Gamma \le L$. Since $\Gamma$ is finitely generated and $\Gamma$ is cocompact in $L$, $L$ is compactly generated. By the last point of the proposition, the subgroup $R := \Res(L)$ is a closed normal subgroup of $G$. Suppose $\Res(G) \leq R$. Then in particular $\Res(G) \leq L$ because $R \leq L$. Then by (\ref{item-not-loc-fg}) we have that $L$ is not locally finitely generated, in contradiction with Proposition  \ref{prop-join-lattice-situation}. So $\Res(G) \leq R$ cannot hold, and therefore by (\ref{item-alternative-normal-subgroups}) there must exist a compact subgroup $K \leq R$ such that $K$ is open in $R$ and $K$ is normal in $G$. In particular $K$ is normal in $L$. One easily verifies that $R/K = \Res(L/K)$. Since $K$ was open in $R$, it follows that $\Res(L/K)$ is discrete. Since $L/K$ is compactly generated, by Theorem \ref{thm-res-disc-res-triv} there is a compact open normal subgroup $K'/K$ of $L/K$; thus $K'$ is a compact open normal subgroup of $L$.  Since $\Gamma$ is cocompact in $L$, we see that $K'\Gamma$ is a finite index subgroup of $L$. Now the subgroup $L$ being cocompact in $G$, it inherits the property of admitting a compact open subgroup with trivial normal core (Lemma  \ref{lem-faith-Cay-Ab-coco}). All together we have that $\Gamma,H,L$ satisfy all the properties of Proposition \ref{prop-latt-cpct-by-discrete}. The conclusion follows. 
\end{proof}

\begin{Remark}
It is worth mentioning recent examples by Huang-Mj that suggest Theorem \ref{thm-main-not-loc-fg} seem rather optimal in that level of generality, even for rich ambient tdlc group $G$. Theorem 1.5 and Corollary 1.6 in \cite{Huang-Mj} exhibit examples of right angled Artin groups $\Gamma$, viewed as a cocompact lattice in $G$ the automorphism group of the cube complex associated to $\Gamma$, a subgroup $H$ of $\Gamma$ that is finitely generated, infinite and infinite index in $\Gamma$, and a non-discrete subgroup $C$ of $G$ with $\Gamma \leq C \leq \Comm_G(\Gamma)$ such that $H$ is commensurated in $C$. 
\end{Remark}

\section{The proof of Theorem \ref{thm-intro-main-completion}} \label{sec-local-assumption-completion}

\subsection{An auxiliary result}

The goal of this subsection is to prove Theorem  \ref{thm-local-contraint-completion}. The setting is rather abstract: unlike elsewhere in the paper, the groups $\Gamma$ and $C$ are not assumed to be respectively discrete and dense in some ambient locally compact group. 

Recall that a locally compact group is  \textbf{locally elliptic} if every compact subset is contained in a compact subgroup. Every locally compact group $G$ admits a unique locally elliptic closed normal subgroup containing any locally elliptic closed normal subgroup of $G$. It is called the locally elliptic radical (LE-radical) of $G$, and is denoted $\RadLE(G)$. It is a topologically characteristic subgroup of $G$. 

\begin{Definition}
The local prime content of a profinite group $U$ is the set of primes $p$ such that $p$ divides the order of every open subgroup of $U$. It is denoted $\lpc(U)$.
\end{Definition}

An equivalent formulation is that $p$ belongs to $\lpc(U)$ if and only if $U$ contains an infinite pro-$p$ subgroup.

\begin{Theorem} \label{thm-local-contraint-completion}
	Let $\Gamma$ be a commensurated subgroup of a group $C$. Suppose that: \begin{enumerate}
		\item \label{item-normalizer-fi-ind} for every subgroup $\Gamma' $ commensurable with $\Gamma$, $\Gamma' $ has finite index in $N_C(\Gamma')$. 
		\item \label{item-LERad-fi} The LE-radical of the group $C / \! \! /\Gamma$ is finite; 
		\item \label{item-local-assumption-completion} $C / \! \! /\Gamma$ admits a compact open subgroup $U$ such that: \begin{enumerate}
			\item \label{item-inf-loc-primes} $\lpc(U)$ is infinite;
			\item \label{item-assu-local-normal}  there is no pair $(M,N)$ of infinite closed normal subgroups of $U$ such that $M \cap N$ is finite and $\lpc(U/N)$ is finite.
		\end{enumerate} 
	\end{enumerate}
	Then every finitely generated commensurated subgroup of $C$ is virtually contained in $\Gamma$.
\end{Theorem}

Note that assumption (\ref{item-LERad-fi}) is a global condition on the completion $C / \! \! /\Gamma$, while assumptions (\ref{item-local-assumption-completion}) are local conditions.

\begin{proof}
	Let $H$ be a finitely generated commensurated subgroup of $C$. The subgroup $\tau_{C,\Gamma}^{-1}(U)$ is commensurable with $\Gamma$. Let $\Lambda$ be a representative of the join of $[\Gamma]$ and $[H]$ of the form $\Lambda = \left\langle \tau_{C,\Gamma}^{-1}(U), H' \right\rangle $ with $H'$ a finite index subgroup of $H$. Such a $\Lambda$ exists by case \ref{A2-fg} of Theorem~\ref{thm:commensurated_join}. Let $E$ be the closure of $ \tau_{C,\Gamma}(\Lambda)$ in $C / \! \! /\Gamma$. The group $E$ is generated by $U$ and $ \tau_{C,\Gamma}(H)$, and hence is compactly generated. Since $\Lambda$ is commensurated in $C$ and $\tau_{C,\Gamma}(C)$ is dense in $C / \! \! /\Gamma$, $E$ is commensurated in $C / \! \! /\Gamma$. Therefore $R:= \Res(E)$ is normal in $C / \! \! /\Gamma$ by Lemma  \ref{lem-commens-normalizes-res}. 
	
	Let $K = \Core_E(U)$. As observed in \cite[Proposition 4.6]{CRW-part-II}, whenever $G$ is a compactly generated tdlc group and $V$ is a compact open subgroup, $\lpc(V / \Core_G(V))$ is finite. So $\lpc(U/K)$ is finite. Let $K_R:=K \cap R$. The subgroup $K_R$ is a compact normal subgroup of $R$, and hence $K_R \leq \RadLE(R)$. Since $R$ is normal in $C / \! \! /\Gamma$ and $\RadLE(R)$ is topologically characteristic in $R$, we have that $\RadLE(R)$ is normal in $C / \! \! /\Gamma$. Hence $\RadLE(R)$ is contained in $\RadLE(C / \! \! /\Gamma)$. By (\ref{item-LERad-fi}) the latter is finite, so we deduce $\RadLE(R)$ is finite, and then $K_R$ is finite. Since $R$ is normal in $C / \! \! /\Gamma$, $U \cap R$ is normal in $U$. Since $K_R$ is the intersection between $K$ and $U \cap R$, it follows from (\ref{item-assu-local-normal}) that one of $K$ or $U \cap R$ is finite. The subgroup $K$ is infinite because $\lpc(U)$ is infinite by (\ref{item-inf-loc-primes}) and $\lpc(U/K)$ is finite. So $U \cap R$ is finite. So we deduce $R$ is discrete. Since $E$ is compactly generated, by Theorem \ref{thm-res-disc-res-triv} this implies that compact open normal subgroups of $E$ form a basis of identity neighborhoods. Let $V$ be a compact open normal subgroup of $E$, and $\Gamma' =  \tau_{C,\Gamma}^{-1}(V)$. Since $V$ is normal in $E$ and $\tau_{C,\Gamma}(H) \leq E$, we have $H \leq N_C(\Gamma')$. The subgroup $\Gamma'$ being commensurable with $\Gamma$,  by (\ref{item-normalizer-fi-ind}) the subgroup $N_C(\Gamma')$ contains $\Gamma'$ as a finite index subgroup. In particular $H$ is virtually contained in $\Gamma'$, and hence in $\Gamma$ as well. 
\end{proof}

\subsection{Maximal commensurated subgroups}

For Theorem  \ref{thm-local-contraint-completion} to be applicable, one needs to have at one’s disposal sufficient conditions on a pair $(C,\Gamma)$ ensuring that the associated Schlichting completion has finite LE-radical. The purpose of the present subsection as well as the next one is to provide such conditions.  

\begin{Proposition} \label{prop-completion-max-co}
	Let $C$ be a group, and let $\Gamma_m$ a commensurated subgroup of $C$ such that $\Gamma_m$ is maximal in its commensurability class in $C$. Let $U_m$ be the closure of $\tau_{C,\Gamma_m}(\Gamma_m)$ in  $C / \! \! /\Gamma_m$. Then: \begin{enumerate}
		\item \label{item-Um-max-co} $U_m$ is a maximal compact open subgroup of $C / \! \! /\Gamma_m$;
		\item  \label{item-completion-LE-trivial} $C / \! \! /\Gamma_m$ has trivial LE-radical;
		\item \label{item-completion-QZ-trivial} Suppose in addition that $\Gamma$ has finite index in $N_C(\Gamma)$ for every $\Gamma$ commensurable with $\Gamma_m$. Then $C / \! \! /\Gamma_m$ has trivial quasi-center. 
	\end{enumerate}
\end{Proposition}

\begin{Lemma} \label{lem-max-co-QZ}
	Let $L$ be a tdlc group with a maximal compact open subgroup $U_m$, and with the property that for every open normal subgroup $V$ of $U_m$, $N_L(V) = U_m$. Let $K$ be a compact normal subgroup of $L$, and write $Q = L/K$ and $\pi: L \to Q$ the canonical projection. Then $\pi^{-1}\left(\QZ(Q) \right) \leq \Core_G(U_m)$.
\end{Lemma}

\begin{proof}
Since $\pi^{-1}\left(\QZ(Q) \right) $ is  normal in $L$, it suffices to show that $\pi^{-1}\left(\QZ(Q) \right) \leq U_m$. So let $g \in \pi^{-1}\left(\QZ(Q) \right)$. By lifting to $L$ a compact open subgroup of $Q$ contained in the centralizer of $\pi(g)$, we obtain a compact open subgroup $V$ of $L$ such that $V$ contains $K$ and $[g,V] \leq K$. In particular $g$ normalizes $V$. Upon reducing $V$ we can assume that $V$ is contained in $U_m$ and $V$ is normal in $U_m$.  By our assumption this implies $N_L(V) = U_m$. Therefore $g \in U_m$, as desired. 
\end{proof}

\begin{Lemma} \label{lem-LE-rad-compact}
	Let $L$ be a tdlc group with a maximal compact open subgroup $U_m$. Then $\RadLE(L)$ is contained in $U_m$.
\end{Lemma}

\begin{proof}
	Consider the subgroup $O = \RadLE(L) U_m$. Since being locally elliptic is stable under extensions, $O$ remains locally elliptic. So for $g \in \RadLE(L)$, the subgroup generated by $g$ and $U_m$ must be compact. By maximality it has to be equal to $U_m$, and hence $g \in U_m$.
\end{proof}

\begin{proof}[Proof of Proposition  \ref{prop-completion-max-co}]
(\ref{item-Um-max-co}) follows from the fact that the pre-image in $C$ of a compact open subgroup of $C / \! \! /\Gamma_m$ is commensurable with $\Gamma_m$. By Lemma \ref{lem-LE-rad-compact} the LE-radical of $C / \! \! /\Gamma_m$ is contained in $U_m$. Since $U_m$ has trivial normal core in $C / \! \! /\Gamma_m$, conclusion (\ref{item-completion-LE-trivial}) follows. 

Under the additional assumption that the commensurability class of $\Gamma_m$ in $C$ is stable under taking normalizers, it follows that the normalizer in $C / \! \! /\Gamma_m$ of any open normal subgroup $V$ of $U_m$ is exactly $U_m$. (\ref{item-completion-QZ-trivial}) therefore follows from Lemma \ref{lem-max-co-QZ} using again that $U_m$ has trivial normal core in $C / \! \! /\Gamma_m$. 
\end{proof}

\begin{Remark}
The combination of Proposition  \ref{prop-completion-max-co} and 	Theorem \ref{thm-local-contraint-completion} yields a method to show the conclusion of Theorem \ref{thm-local-contraint-completion} holds true, provided the local conditions on the completion in Theorem \ref{thm-local-contraint-completion} can be checked. We do not develop further this approach in the present paper. Let us just mention that, in the case of automorphism groups of regular trees, based on the existence of maximal cocompact lattices in $\Aut(T_d)$ (as follows from \cite[Theorem 1.4]{Trofimov-Weiss-95}) together with the CSP discussed in \S  \ref{subsec:U(F)}, this approach can also lead to a proof of Theorem \ref{thm-main-trees}. 
\end{Remark}

\subsection{The proof of Theorem \ref{thm-intro-main-completion}}

The first goal of this subsection is to  provide another way to ensure assumption (\ref{item-LERad-fi}) in Theorem  \ref{thm-local-contraint-completion} holds true (see Proposition~\ref{prop:LE-radical-restriction} and
Remark \ref{rmk-non-cpt-LE}). Unlike in the previous subsection, the approach here does not rely on any maximality assumption. Instead, we make use of the following result from \cite{CLB-covol}. For a locally compact group $G$, we denote by $\Sub(G)$ the space of closed subgroups of $G$, equipped with the Chabauty topology (see e.g.\ \cite{Schochetman-71}). A subgroup $K$ of a tdlc group $G$ is called l\textbf{ocally normal} in $G$ if there is a compact open subgroup $U$ of $G$ such that $U$ normalizes $K$.

\begin{Proposition}[{See \cite[Proposition 3.6]{CLB-covol}}]\label{prop-approx-disc-subgroups}
Let $G$ be a compactly generated tdlc group that admits a compact open subgroup with trivial normal core, and such that $\QZ(G)$ is discrete. Suppose $G$ is non-discrete, and there exists a sequence $(\Gamma_n)$ of discrete subgroups of $G$ that converges in $\Sub(G)$ to a finite index subgroup of $G$. Then $G$ admits an infinite compact locally normal subgroup $K$ that is pro-$p$ for some prime $p$. 
\end{Proposition}

The following lemma is a variation of \cite[Lemma 3.14]{CRW-part-II}. We say a tdlc group $G$  is \textbf{locally pro-$p$} if some compact open subgroup of $G$ is pro-$p$. 

\begin{Lemma} \label{lem-loc-normal-pro-p-loc-pro-p}
	Let $G$ be a $\sigma$-compact tdlc group, $K \leq G$ a compact locally normal subgroup of $G$ such that the abstract normal subgroup of $G$ generated by $K$ is open in $G$. Then there is a family $\left\lbrace L_1, \ldots, L_n \right\rbrace $ of open subgroups of conjugates of $K$ that normalize each other and such that $L_1 \cdots L_n$ is a compact open subgroup of $G$. In particular if $K$ is pro-$p$ then $G$ is locally pro-$p$. 
\end{Lemma}

\begin{proof}
Since $G$ is $\sigma$-compact and $K$ is locally normal, there are countably many conjugates $(K_n)$ of $K$ in $G$. The subgroup they generate is open by assumption, so by the Baire category theorem there are $K_1, \ldots, K_n$ such that the subset $K_1 \cdots K_n$ has non-empty interior. Let $U$ be a compact open subgroup of $G$ that normalizes $K_i$ for every $i$, and let $L_i = U \cap K_i$. Then the subgroups $L_i$ normalize each other, and hence $L = L_1 \cdots L_n$ is a subgroup of $G$, which normalizes each $K_i$ since $L \leq U$. This implies $K_1 \cdots K_n$ is covered by finitely many cosets of $L$. Therefore $L$ also has non-empty interior, and hence is open. For the last assertion, if $K$ is pro-$p$ then so is every $L_i$, and hence so is $L$.
\end{proof}

\begin{Proposition}\label{prop:LE-radical-restriction}
Let $\Gamma$ be a finitely generated cocompact lattice of a tdlc group $G$, and $C$ a dense subgroup of $G$ such that $\Gamma \leq C \leq \Comm_G(\Gamma)$. Suppose that $\QZ(G)$ is discrete, $\Res(G)$ is open and non-compact, and every abstract normal subgroup of $G$ is contained in $\QZ(G)$ or contains $\Res(G)$. Suppose also the group $C / \! \! /\Gamma $ admits an open subgroup $O$ that is locally elliptic, commensurated, and non-compact. Then there is a prime $p$ such that $G$ is locally pro-$p$ and $G$ is not locally finitely generated. 
\end{Proposition}

\begin{proof}
Denote by $\tau_{C,\Gamma} : C \to C/\! \!/\Gamma$ the homomorphism from $C$ to $C/\! \!/\Gamma$. Consider an open and commensurated subgroup $O$ of $C/\! \!/\Gamma$ that is written as an increasing union $O = \bigcup O_n$ of compact open subgroups of $C/\! \!/\Gamma$. Set $\Gamma_n = \tau_{C,\Gamma}^{-1}(O_n)$ and $\Lambda = \tau_{C,\Gamma}^{-1}(O)$. Then $(\Gamma_n)$ forms a sequence of cocompact lattices of $G$ ascending to $\Lambda$, and $\Lambda$ is commensurated in $C$ because $O$ is commensurated in $C/\! \!/\Gamma$. Let $L := \overline{\Lambda}$ be the closure of $\Lambda$ in $G$. Since $(\Gamma_n)$ is ascending, $L$ is equal to the limit of $(\Gamma_n)$ in $\Sub(G)$ \cite[Theorem I]{Schochetman-71}. Since $O$ is not compact, no member of the sequence $\Gamma_n$ has finite index in $\Lambda$, and $L$ is non-discrete. Also $L$ is compactly generated because $L$ is cocompact in $G$.

Since $\Lambda$ is commensurated in $C$, the subgroup $L$ is also commensurated by $C$. Hence $C$ normalizes $\Res(L)$ by Lemma  \ref{lem-commens-normalizes-res}, and $C$ being dense in $G$ and $\Res(L)$ being closed, $\Res(L)$ is normal in $G$. Since $L$ is compactly generated, if $\Res(L)$ is discrete then by Theorem \ref{thm-res-disc-res-triv}  one deduces that $L$ is compact-by-discrete. Using Proposition \ref{prop-coc-cpt-normal} and the fact that compact normal subgroups of $G$ are finite (consequence of the present assumptions), we deduce $L$ is discrete, which is a contradiction. Therefore $\Res(L)$ cannot be discrete, and therefore it is open since $\Res(L)$ is normal in $G$. In particular $L$ is open, and therefore $L$ has finite index in $G$. Therefore we are in the situation of Proposition \ref{prop-approx-disc-subgroups}. We deduce that $G$ has an infinite compact locally normal subgroup $K$ that is pro-$p$. Since $K$ is compact infinite, the abstract normal subgroup of $G$ generated by $K$ is non-discrete. Therefore it is open, and we are in position to apply Lemma  \ref{lem-loc-normal-pro-p-loc-pro-p}. We deduce $G$ has a compact open subgroup $U$ that is pro-$p$. From here, arguing as in the end of the proof of Theorem F in \cite{CLB-covol}, we deduce $U$ cannot be  finitely generated. We repeat the argument for completeness. Upon replacing $U$ by $U \cap L$ we can assume $U \leq L$. Since $U$ is open, taking the intersection with $U$ defines a continuous map on $\Sub(G)$ \cite[Proposition 3]{Schochetman-71}. Hence $(\Gamma_n \cap U)$ converges to $U$. Recall that finitely generated pro-$p$ groups have open Frattini subgroup. Hence if $U$ were finitely generated we would have $U = \Phi(U) (\Gamma_n \cap U)$ for $n$ large enough, and hence  $U = \Gamma_n \cap U$. Since $U$ is open and $\Gamma_n$ is discrete, this is absurd. 
\end{proof}

\begin{Remark} \label{rmk-non-cpt-LE}
The existence in $C / \! \! /\Gamma $ of an open subgroup $O$ that is locally elliptic, commensurated and non-compact, holds true whenever $C / \! \! /\Gamma $ has a non-compact LE-radical. Indeed, $O = \RadLE(C / \! \! /\Gamma ) U$, where $U$ is a compact open subgroup of $C / \! \! /\Gamma $, is such a subgroup. 
\end{Remark}

The following result is the main result of this subsection. It combines the approach of Section \ref{sec-not-loc-fg} together with Theorem  \ref{thm-local-contraint-completion} and  Proposition~\ref{prop:LE-radical-restriction}. 

\begin{Theorem} \label{thm-combined-arguments}
Let $\Gamma$ be a finitely generated cocompact lattice of a tdlc group $G$, and $C$ a dense subgroup of $G$ such that $\Gamma \leq C \leq \Comm_G(\Gamma)$. Suppose $\QZ(G)$ is discrete, $\Res(G)$ is open, and every abstract normal subgroup of $G$ is contained in $\QZ(G)$ or contains $\Res(G)$. Suppose also the $C$-congruence completion $\widehat{\Gamma}^C$ of $\Gamma$ has the following properties:
	\begin{enumerate}
		\item \label{item-many-primes} $\lpc(\widehat{\Gamma}^C)$ is infinite;
		\item \label{item-no-commuting-normal}  there is no pair $(M,N)$ of infinite closed normal subgroups of $\widehat{\Gamma}^C$ such that $M \cap N$ is finite and $\lpc(\widehat{\Gamma}^C/N)$ is finite.
		\item \label{item-no-finite-exponent-normal} there is no infinite closed normal subgroup $M$ of $\widehat{\Gamma}^C$ such that $M$ embeds as a closed subgroup of $F^I$ for some finite group $F$ and set $I$.
	\end{enumerate}
	Then every finitely generated commensurated subgroup of $C$ is virtually contained in $\Gamma$.
\end{Theorem}

\begin{proof}
By Lemma \ref{lem-simplicity-assumptions-faith-Cay-Ab} $G$ admits a compact open subgroup with trivial normal core. Hence in the situation where $G$ is compact-by-discrete the conclusion holds thanks to Proposition \ref{prop-latt-cpct-by-discrete}.  Hence for the rest of the proof we assume $G$ is not compact-by-discrete. In particular $\Res(G)$ is not compact. 
	
	If $G$ is not locally finitely generated, the conclusion follows by Theorem \ref{thm-intro-main-not-loc-fg}.  So we may assume $G$ is locally finitely generated. By Remark \ref{rmk-non-cpt-LE},  Proposition~\ref{prop:LE-radical-restriction} ensures that $\RadLE(G_2)$ is compact, where $G_2 := C / \! \! /\Gamma$. Since $G_2$ admits a compact open subgroup with trivial normal core, by Lemma \ref{lem-faith-Cay-Ab-variety} there is a finite group $F$ and a set $I$ such that $\RadLE(G_2)$ is isomorphic to a closed subgroup of $F^I$. Hence it follows from (\ref{item-no-finite-exponent-normal}) that $\RadLE(G_2) \cap \widehat{\Gamma}^C$ is finite (recall that we identify $\widehat{\Gamma}^C$ with the image closure of $\Gamma$ in $G_2$). Therefore $\RadLE(G_2)$ is finite, and hence Theorem \ref{thm-local-contraint-completion} is applicable. The condition on normalizers is indeed satisfied here by Proposition \ref{prop-latt-fi-normalizer}. The conclusion follows.
\end{proof}

\begin{Remark} \label{rmq-normal-lpc-infinite-sufficient}
For $M,N$ such that $M \cap N$ is finite and $\lpc(\widehat{\Gamma}^C/N)$ is finite, $\lpc(M)$ must be finite because $\lpc(M) \subseteq \lpc(\widehat{\Gamma}^C/N)$. Also $\lpc(M)$ is finite whenever $M$ is as in (\ref{item-no-finite-exponent-normal}). Hence the condition \textit{every infinite closed normal subgroup $M$ of $\widehat{\Gamma}^C$ is such that $\lpc(M)$ is infinite} is enough to ensure  (\ref{item-no-commuting-normal}) and (\ref{item-no-finite-exponent-normal}).
\end{Remark}

\section{Groups of automorphisms of trees } \label{sec-trees}

\subsection{Preliminaries}

Let $T$ be a locally finite tree. The group $\Aut(T)$ is equipped with the compact-open topology for the action on $T$. If $G$ is a subgroup of $\Aut(T)$, denote by $G^+$ the subgroup of $G$ generated by fixators of edges. This is a normal subgroup of $G$. Note that if $G$ is equipped with the induced topology from $\Aut(T)$, then $G^+$ is an open subgroup of $G$.
We refer to \cite{Tits-arbres} for the definition of the independence property $(P)$.

\begin{Theorem}[{Tits \cite[Théorème 4.5]{Tits-arbres}}]\label{thm-Tits}
Suppose the action of $G \leq \Aut(T)$ on $T$ is minimal and has no fixed end, and $G$ has Tits' independence property $(P)$. Then every non-trivial  normal subgroup of $G$ contains $G^+$. 
\end{Theorem}

We will need the following result, which follows from M.\ Hall's theorem that any finitely generated subgroup of a free groups is a free factor of a finite index subgroup. 

\begin{Proposition} \label{prop-commens-free}
	Let $F$ be a finitely generated free group, and $H$ an infinite finitely generated commensurated subgroup.  Then $H$ has finite index in $F$.
\end{Proposition}

\subsection{The proof of Theorem \ref{thm-intro-main-trees}}

Let $T$ be a locally finite tree. Fix an edge $e$ of $T$, and let $\Aut(T)_e$ be the fixator of $e$ in $\Aut(T)$. Consider the continuous homomorphism $\Aut(T)_e \to \prod_{n \geq 1} \left\lbrace \pm 1\right\rbrace $ which associates to an element of $\Aut(T)_e$ the sequence of signatures of the permutations induced by $g$ on the spheres around $e$. A basic observation from  \cite{Mozes-ICM, Burger-Mozes-Zimmer--linear-rep} is that when $T$ is regular of degree $ \geq 3$, this homomorphism is surjective. Since an infinite profinite group of exponent $2$ is never finitely generated, this implies that the profinite group $\Aut(T)_e$ is not finitely generated. Hence $\Aut(T)$  is not locally finitely generated. When $T$ is no longer regular, $\Aut(T)_e \to \prod_{n \geq 1} \left\lbrace \pm 1\right\rbrace $ need not be surjective, but it is not difficult to check that it has infinite image provided $ \Aut(T)$ is non-discrete. This yields:

\begin{Lemma} \label{lem-tree-not-loc-fg}
	Let $T$ be a locally finite tree such that $ \Aut(T)$ is non-discrete. Then $ \Aut(T)$ is not locally finitely generated. 
\end{Lemma}

\begin{Theorem}[{Bass--Kulkarni for bi-regular trees \cite{Bass-Kulk-90}, Liu \cite{Liu-density}}] \label{thm-dense-tree}
Let $T$ be a locally finite  tree such that  $\Aut(T)$ acts cocompactly on $T$. Then every cocompact lattice of $\Aut(T)$ has dense commensurator. 
\end{Theorem}

\begin{Theorem} \label{thm-main-trees}
	Let $T$ be a locally finite  tree such that  $\Aut(T)$ acts cocompactly on $T$. Let $\Gamma$ be a cocompact lattice of $G = \Aut(T)$. Then, up to commensurability, $\Gamma$ is the only infinite finitely generated commensurated subgroup of $\Comm_G(\Gamma)$.
\end{Theorem}

\begin{proof}
	Observe that since $G$ acts cocompactly on $T$, $G$ admits a compact open subgroup with trivial normal core. If $G$ admits a compact open normal subgroup then the statement follows from  Proposition \ref{prop-latt-cpct-by-discrete}. Hence we may assume $G$ is not of this form. In particular, $T$ has infinitely many ends. Since $\Gamma$ is a discrete subgroup of $G$ acting cocompactly on $T$,  $\Gamma$ does not fix any end of $T$, so neither does $G$. Let $H$ be an infinite finitely generated commensurated subgroup of $\Comm_G(\Gamma)$. We wish to show first that $H$ is virtually contained in $\Gamma$.
	
	Assume for a moment that the $G$-action on $T$ is minimal. Since $G = \Aut(T)$ satisfies Tits' independence property $(P)$, every non-trivial normal subgroup of $G$ contains $G^+$ by Theorem \ref{thm-Tits}. In particular every non-trivial normal subgroup of $G$ is open. By Lemma \ref{lem-tree-not-loc-fg} $G$ is not locally finitely generated, and by Theorem \ref{thm-dense-tree} the subgroup $\Comm_G(\Gamma)$ is dense in $G$. Hence all the assumptions of Theorem \ref{thm-intro-main-not-loc-fg} are satisfied, and we deduce $H$ is virtually contained in $\Gamma$.

In general the $G$-action on $T$ need not be minimal, but $G$ admits a unique minimal invariant subtree $T_{min}$. Let $K$ be the kernel of the action of $G$ on $T_{min}$ and $\overline{G}$  the image  of $G$ in $\Aut(T_{min})$.  The subgroup $\overline{G}$ is not necessarily the entire $\Aut(T_{min})$, but it keeps all the properties that we used about $G$ in the previous paragraph. Hence we deduce $KH$ is virtually contained in $K \Gamma$. Since $K$ is compact, Proposition \ref{prop-latt-cpct-by-discrete} ensures $H$ is virtually contained in $\Gamma$.
 
To conclude, observe that since $\Gamma$ acts properly and cocompactly on a tree, upon passing to a finite index subgroup we can assume $\Gamma$ is a free group. By the previous paragraph $H \cap \Gamma$ has finite index in $H$. Hence $H \cap \Gamma$ is finitely generated and infinite. On the other hand  $H \cap \Gamma$ is a commensurated subgroup of $\Gamma$. Proposition \ref{prop-commens-free} then asserts $H \cap \Gamma$ has finite index in $\Gamma$.
\end{proof}

\begin{Remark}
In the setting of Theorem \ref{thm-main-trees}, cocompact lattices exist in $\Aut(T)$ if and only if $\Aut(T)$ is a unimodular group, if and only if the tree $T$ covers a finite graph \cite[Corollary 4.10]{Bass-Kulk-90}.
\end{Remark}

\subsection{Groups of automorphisms of trees with prescribed local action} \label{subsec:U(F)}

We now turn to an application of Theorem~\ref{thm-combined-arguments} to closed cocompact subgroups of $\Aut(T)$.

\begin{Corollary}\label{cor-combined-arguments-tree}
Let $T$ be a locally finite tree, $G$ a closed subgroup of $\Aut(T)$ that acts cocompactly on $T$ and $\Gamma$ a cocompact lattice of $G$.  Suppose the following conditions hold:
\begin{enumerate}[(A)]
\item\label{item-dense} $\Comm_G(\Gamma)$ is dense in $G$;
\item\label{item-normal} $\QZ(G)$ is discrete, $\Res(G)$ is open, and every abstract normal subgroup of $G$ is contained in $\QZ(G)$ or contains $\Res(G)$;
\item\label{item-CSP} $\Gamma$ has the CSP in $\Comm_G(\Gamma)$.
\end{enumerate}
Then, up to commensurability, $\Gamma$ is the only infinite finitely generated commensurated subgroup of $\Comm_G(\Gamma)$.
\end{Corollary}

\begin{proof}
	The case where $T$ is bounded is trivial, so we assume $T$ unbounded.  Since  $\Gamma$ acts properly and cocompactly on a tree, $\Gamma$ has a finite index subgroup that is a non-trivial free group. Hence the profinite completion $\widehat{\Gamma}$ has a finite index subgroup that is a non-trivial  free profinite group. Therefore $\widehat{\Gamma}$ verifies the conditions (\ref{item-many-primes}), (\ref{item-no-commuting-normal}), (\ref{item-no-finite-exponent-normal}) from Theorem~\ref{thm-combined-arguments}. As in the previous subsection, since $\Gamma$ is virtually a free group, combining the conclusion of Theorem~\ref{thm-combined-arguments} with  Proposition \ref{prop-commens-free} one obtains the conclusion.
\end{proof}

In the rest of this subsection we prove Theorem~\ref{thm-intro-U(F)} from the introduction using Corollary \ref{cor-combined-arguments-tree}. Let $d \geq 3$, and we denote by $T_d$ the $d$-regular tree. Given a vertex $v$, write $\mathrm{St}(v)$ for the set of edges of $T$ around $v$. In all this subsection we fix a coloring $c$ of the edges of $T_d$ using colors $\{1,2,\dots,d\}$ that is locally bijective, meaning that for every $v$ all edges of $\mathrm{St}(v)$ have different colors. For $g \in \Aut(T_d)$ and $v \in V(T_d)$, the {local action} of $g$ at $v$ is the unique permutation $\sigma(g,v) \in \mathrm{Sym}(d)$ such that $c(ge) = \sigma(g,v)(c(e))$ for all $ e \in \mathrm{St}(v)$.  

Given a permutation group $F \le \mathrm{Sym}(d)$, the group $U(F)$ consists of automorphisms of $T_d$ all of whose local actions belong to $F$ \cite{BM00a}: \[ U(F):= \{g \in \Aut(T_d) \mid \forall v \in V(T_d) : \sigma(g,v) \in F\}.\] It is a closed subgroup of $\Aut(T_d)$ acting transitively on vertices. The group $U(F)$ is non-discrete if and only if the permutation group $F$ does not act feely on $\{1,2,\dots,d\}$.

 The result that $\Aut(T_d)$ admits only one commensurability class of cocompact lattices up to conjugation \cite{Leighton-finite-cover, Bass-Kulk-90} also holds true in the group $U(F)$: 

\begin{Proposition}\label{prop-U(F)-coc-conjugate}
	Let $\Gamma_1, \Gamma_2$ be cocompact lattices in $U(F)$.  Then there is $g \in U(F)$ such that $g\Gamma_1g^{-1}$ and $\Gamma_2$ are commensurable.
\end{Proposition}

\begin{proof}
	The statement can be deduced from a mild adaptation of the proof of Theorem 5.2 from Bass' original article \cite[5.4]{BassCovering} (see Referee's proof of Theorem 5.2). Alternatively, the statement also follows from a more general result proven by Shepherd and Gardam--Woodhouse  \cite[Theorem B-A.1]{Shepherd-GGD-22}.
\end{proof}

In particular any two cocompact lattices of $U(F)$ have their commensurators that are conjugate in $U(F)$. In the sequel it will be convenient to work with a specific one, namely the subgroup $W$ consisting of automorphisms $g$ of $T_d$ such that $\sigma(g,v) = 1$ for every vertex $v$. The group $W$ acts freely transitively on vertices of $T_d$. We denote by $C_d$ the commensurator of $W$ in $\Aut(T_d)$, and by $C_{d,F} = C_d \cap U(F)$ the commensurator of $W$ in $U(F)$. 

\begin{Proposition} \label{prop-dense-comm-U(F)}
	Any cocompact lattice of $U(F)$ has a dense commensurator in $U(F)$.
\end{Proposition}

\begin{proof}
	Lubotzky--Mozes--Zimmer's  Proposition~2.6 from \cite{LMZ-superrigidity} shows $C_d$ is dense in $\Aut(T_d)$. The exact same argument applies within the group $U(F)$, and shows $C_{d,F}$ is dense in $U(F)$. In view of Proposition \ref{prop-U(F)-coc-conjugate}, this yields the conclusion. 
\end{proof}

We recall the description of elements of $C_d$ from \cite{LMZ-superrigidity}. Let $Y$ be a finite, connected, $d$-regular graph (we often identify $Y$ and its set of vertices). We do not allow an edge to be a loop. We allow multiple edges. Fix a coloring $\overline{c}$ of the edges of $Y$ by $\{1,2,\dots,d\}$ that is locally bijective. Given a vertex $v_0$ of $T_d$ and a vertex $y_0$ of $Y$, there is a unique color preserving map $\pi: T_d \to Y$ such that $\pi(v_0) = y_0$, and it is surjective. Let $(\sigma_y)_{y \in Y}$ be a collection  of elements of $\mathrm{Sym}(d)$. We say that it satisfies the compatibility condition if $\sigma_y(\overline{c}(e)) = \sigma_{y'}(\overline{c}(e))$ for every edge $e$ of $Y$ joining two vertices $y$ and $y'$ of $Y$. We refer to Section 2 in \cite{LMZ-superrigidity} for the following result. 

\begin{Proposition} \label{prop-description-LMZ}
	Retain the above notation. Then: \begin{enumerate}
		\item \label{item-charact-element-commens} For $g \in \Aut(T_d)$ and $\gamma \in W$, we have $\gamma \in W \cap g^{-1} W g$ if and only if $\sigma(g, \gamma v) = \sigma(g, v)$ for every vertex $v$ of $T_d$.
		\item \label{item-defines-element-commens} If $(\sigma_y)_{y \in Y}$ satisfies the compatibility condition, then there is a unique $g \in \Aut(T_d)$ such that $\sigma(g,v) = \sigma_{\pi(v)}$ for every vertex $v$, and we have $g \in C_d$. 
	\end{enumerate}
\end{Proposition}

Mozes showed that when $G = \mathrm{Aut}(T_d)$ is the full automorphism group of $T_d$ and $\Gamma$ is a cocompact lattice in $G$, then $\Gamma$ has the CSP in $\Comm_G(\Gamma)$ \cite[Theorem 1.2]{Mozes-CSP-tree}. The argument can be generalized to obtain: 

\begin{Theorem} \label{thm-CSP-U(F)}
	Suppose $F$ does not act freely on $\{1,2,\dots,d\}$. Let $\Gamma$ be a cocompact lattice in $G = U(F)$. Then $\Gamma$ has the CSP in $\Comm_G(\Gamma)$. 
\end{Theorem}

\begin{proof}
	According to Proposition \ref{prop-U(F)-coc-conjugate}, one can assume $\Gamma$ is commensurable with $W$. Since virtually free groups verify the conditions of Lemma \ref{lem-CSP-one-all}, and every subgroup commensurable with $\Gamma$ has no non-trivial finite normal subgroup, it is enough to show $W$ has the CSP in $C_{d,F}$. Let $\Lambda$ be a finite index subgroup of $W$. We will find $g \in C_{d,F}$ such that $W \cap g^{-1} W g \leq \Lambda$. It will be convenient to assume $\Lambda$ preserves the bi-partition of vertices of $T_d$. Again by Lemma \ref{lem-CSP-one-all} there is no loss of generality assuming this. Let $Y$ denote the finite connected $d$-regular graph $ \Lambda \backslash T_d$. We fix a vertex $v_0$ of $T_d$, and denote by $y_0$ the image of $v_0$ in $Y$. We also fix $i \in \{1,2,\dots,d\}$  and $\rho \in F$ a non-trivial element fixing $i$. 
	
	First assume that the edge $e$ colored $i$ between $y_0$ and its $i$-neighbour is a separating edge of $Y$, meaning that removing $e$ yields a graph with two connected components. Call $\mathcal{C}_1$ and $\mathcal{C}_2$ these two  components, and assume $y_0 \in \mathcal{C}_1$. Let $(\sigma_y)_{y \in Y}$ be the collection of permutations defined by $\sigma_y = 1$ for $y \in \mathcal{C}_1$, and $\sigma_y = \rho$ for $y \in \mathcal{C}_2$. It satisfies the compatibility condition, and hence defines an element $g$ of $C_d$ fixing $v_0$ by item (\ref{item-defines-element-commens}) in Proposition \ref{prop-description-LMZ}. And by construction $g \in C_{d,F}$. The vertex $y_0$ is the only vertex of $Y$ for which $\sigma_y = 1$ and admitting a neighbour for which $\sigma_y \neq 1$. Therefore any vertex $v$ of $T_d$ such that $\sigma(g,v) = 1$ and there is a neighbour $w$ of $v$ such that $\sigma(g,w) \neq 1$, must lie in the fiber of $y_0$ under $\pi: T_d \to Y$  (which is the $\Lambda$-orbit of $v_0$). Item (\ref{item-charact-element-commens}) in Proposition \ref{prop-description-LMZ} asserts the subgroup $W \cap g^{-1} W g$ preserves the set of vertices $v$ with this property. Therefore the $(W \cap g^{-1} W g)$-orbit of $v_0$ is contained in the $\Lambda$-orbit of $v_0$. Since actions are free, $W \cap g^{-1} W g$ is contained in $\Lambda$, as desired. 
	
	Now assume that the edge colored $i$ between $y_0$ and its $i$-neighbour is non-separating. We define a new graph $Z$ in the following way. We take two copies $Y_0$ and $Y_1$ of $Y$ with base vertices $y_0^0$ and $y_0^1$, and call $z^0$ and $z^1$ the $i$-neighbours of $y_0^0$ and $y_0^1$. We remove the $i$-edge between $y_0^i$ and $z^i$  in each copy, and put a $i$-edge between $y_0^0$ and $z^1$ and between $y_0^1$ and $z^0$. This new graph is connected by the assumption that the removed edge was non-separating in each copy. And $Z$ is a $2$-fold cover of $Y$. Let $\Lambda'$ be the corresponding index two subgroup of $\Lambda$. We will find $g \in C_{d,F}$ such that $W \cap g^{-1} W g \leq \Lambda'$. Let $(\sigma_z)_{z \in Z}$ be the collection of permutations defined by $\sigma_z = 1$ for $z \in Y_0$, and $\sigma_y = \rho$ for $y \in Y_1$. As before, applying item (\ref{item-defines-element-commens})  in Proposition \ref{prop-description-LMZ} (with base vertex $y_0^0$ in $Z$) this yields a well defined element $g \in C_{d,F}$ fixing $v_0$. There are exactly two vertices of $Z$ for which $\sigma_z = 1$ and admitting a neighbour for which $\sigma_z \neq 1$, namely $y_0^0$ and $z^0$. Since these two vertices do not have the same type (type has a meaning in the quotient graph by the assumption that $\Lambda$ preserves the bi-partition  of $T_d$), $y_0^0$ is the only vertex of the given type with the property that we have $\sigma_z = 1$ and admitting a neighbour for which $\sigma_z \neq 1$. As before, using (\ref{item-charact-element-commens})  in Proposition \ref{prop-description-LMZ} we infer $W \cap g^{-1} W g$ is contained in $\Lambda'$. 
\end{proof}

We now conclude the proof of Theorem~\ref{thm-intro-U(F)}.

\begin{Theorem} \label{thm-main-U(F)}
	Let $\Gamma$ a cocompact lattice of $G = U(F)$. Then, up to commensurability, $\Gamma$ is the only infinite finitely generated commensurated subgroup of $\Comm_G(\Gamma)$.
\end{Theorem}

\begin{proof}
The statement follows immediately from Proposition~\ref{prop-commens-free} if $F$ acts freely, so we may assume $F$ does not act freely.  Then $G$ and $\Gamma$ satisfy the conditions of Corollary~\ref{cor-combined-arguments-tree}: \ref{item-dense} by Proposition~\ref{prop-dense-comm-U(F)}, \ref{item-normal} by Theorem~\ref{thm-Tits} since $U(F)$ has Tits' independence property (P), and \ref{item-CSP} by Theorem~\ref{thm-CSP-U(F)}.
\end{proof}

\begin{Remark}
	Whether the group $U(F)$ is locally finitely generated is characterized in terms of properties of $F$ \cite[Corollary 0.2]{BM-U(F)-loc-finite-gen}. In particular some groups $U(F)$ are locally finitely generated, for example $U(\mathrm{Alt}(d))$ for $d \ge 6$. For those the conclusion of Theorem \ref{thm-main-U(F)} cannot be obtained thanks to Theorem \ref{thm-intro-main-not-loc-fg}. 
\end{Remark}

\begin{Remark}
	We expect Corollary~\ref{cor-combined-arguments-tree} to be applicable in a wider setting than Theorem \ref{thm-main-U(F)}. The case of a closed cocompact subgroup $G$ of $\Aut(T)$ with property $(P_k)$ from \cite{BanksElderWillis}, a generalization of property (P), seems to be natural to consider.  For such a $G$, the analogue of Theorem~\ref{thm-Tits} is \cite[Theorem~7.3]{BanksElderWillis}, so condition \ref{item-normal} of Corollary~\ref{cor-combined-arguments-tree} is satisfied.  Condition \ref{item-dense} is also satisfied: the statement of Proposition \ref{prop-U(F)-coc-conjugate} is covered by \cite[Theorem B-A.1]{Shepherd-GGD-22}, and Shepherd's construction of the conjugating element $g$ allows it to be taken from any identity neighborhood in $G$, see \cite[Remark~5.2]{Shepherd-GGD-22}, ensuring the commensurator is dense. Hence if $\Gamma$ has CSP in $G$, then Corollary~\ref{cor-combined-arguments-tree} applies.
\end{Remark}

\section{Groups of automorphisms of right angled buildings} \label{sec-buildings}

We refer to \cite{Davis-buildings, Hag-Pau-arbo} for a more comprehensive treatment of the notions presented here. Let $(W,S)$ be a right-angled Coxeter system. The cardinality of $S$, called the rank, is assumed to be finite. We denote by $\mathcal{G}(W,S)$ the graph with vertex set $S$ and edges $\left\lbrace s,t \right\rbrace $ if and only if $m_{s,t} = 2$, where $(m_{s,t})$ is the Coxeter matrix of $(W,S)$. Conversely every simplicial graph $\mathcal{G}$ with vertex set $S$ defines a right-angled Coxeter system $(W,S)$, with Coxeter matrix  defined by $m_{s,s} = 1$, $m_{s,t} = 2$ if $\left\lbrace s,t \right\rbrace $ is an edge in $\mathcal{G}$ and $m_{s,t} = \infty$ otherwise. Recall that $(W,S)$ is irreducible if there does not exist a non-trivial partition $S = S_1 \cup S_2$ with $m_{s,t} = 2$ for every $s \in S_1, t \in S_2$. Equivalently, the complement graph of $\mathcal{G}(W,S)$ is connected.

For every family $(q_s)_{s \in S}$ of integers $ \geq 2$, there exists a unique up to isomorphism right-angled building $X$ of type $(W,S)$ that is semi-regular of thickness $(q_s)_{s \in S}$, meaning that for every $s \in S$ each panel of type $s$ has thickness $q_s$ \cite[Theorem 5.1]{Davis-buildings}, \cite[Proposition 1.2]{Hag-Pau-arbo}. We say that $X$ is \textbf{thick} if $q_s \geq 3$ for every $s \in S$. We say that $X$ is \textbf{irreducible} if $(W,S)$ is.

We denote by $\Aut(X)$ the group of all automorphisms of $X$, i.e.\ the group of automorphisms of the chamber graph of $X$. The group  $\Aut(X)$ is equipped with the compact open topology for the action on the chamber set of $X$ (viewed as a discrete set). We also denote by $\Aut_0(X)$ the group of type-preserving automorphisms of $X$. This is a closed and cocompact subgroup of $\Aut(X)$.

 Let $P = (P_s)_{s \in S}$ be a family of finite groups, which are all supposed to be non-trivial. The \textbf{graph product } $\Gamma_P$ of the family $P$ associated to $(W,S)$ (or to the graph $\mathcal{G}(W,S)$)  is the quotient of the free product $\ast_S P_s$ obtained by adding relations $[P_s,P_t] = 1$ for every $s,t$ such that $m_{s,t} = 2$. The natural homomorphism $P_s \to \Gamma_P$ is injective for every $s$, and $\Gamma_P$ is  generated by their images. The group $\Gamma_P$ is naturally the chamber system of a semi-regular right-angled building $X$ of type $(W,S)$, whose thickness $(q_s)_{s \in S}$ is $q_s = |P_s|$. We say $\Gamma_P$ is irreducible if $(W,S)$ is. The left action of $\Gamma_P$ on itself induces an action on $X$ by automorphisms. From now on we view $\Gamma_P$ systematically as a cocompact lattice of $\Aut(X)$. The action of $\Gamma_P$ on $X$ is by type-preserving automorphisms, so $\Gamma_P$ actually lies inside $\Aut_0(X)$.

  We set some notation that will be used later on. The set of chambers of the right-angled building $X$ is denoted $\mathrm{Ch}(X)$. More generally if $R$ is a residue of $X$, the chamber set of $R$ is denoted $\mathrm{Ch}(R)$. For $c \in \mathrm{Ch}(X)$ and $R$ a residue of $X$, there is a unique $c' \in \mathrm{Ch}(R)$ that is closest to $c$, called the projection of $c$ on $R$ and denoted $\mathrm{pr}_R(c)$.  If $R,R'$ are residues of $X$, the set of projections $\mathrm{pr}_R(c)$, where $c$ ranges over $\mathrm{Ch}(R')$, forms a residue, denoted $\mathrm{pr}_R(R')$. The residues $R,R'$ are parallel if $\mathrm{pr}_R(R')=R$ and $\mathrm{pr}_{R'}(R) = R'$. For $I \subset  S$ we let $ I^\perp = \left\lbrace s' \in S \, | \, s' \notin I  \,  \text{and}  \, ss' = s's \,  \text{for every}  \, s\in I \right\rbrace$. In the special case $I = \left\lbrace s\right\rbrace $ we write $s^\perp $ instead of $\left\lbrace s\right\rbrace ^\perp $.

 \subsection{Normal and commensurated subgroups}

\begin{Proposition} \label{prop-inside-lattice}
	Let $\Gamma_P$ be an irreducible graph product of finite groups. Let $H$ be an infinite finitely generated commensurated subgroup of $\Gamma_P$. Then $H$ has finite index in $\Gamma_P$. 
\end{Proposition}

\begin{proof}
	Let $(W,S)$, $P = (P_s)_{s \in S}$ as in the definition of $\Gamma_P$, and let $X$ be the right angled building associated to $\Gamma_P$. For simplicity in this proof we write $\Gamma$ instead of $\Gamma_P$. We first show that $H$ cannot stabilize a proper residue of $X$. Assume for a contradiction that $R$ is a residue of type $I \subsetneq S$ stabilized by $H$, and choose $R$ such that $R$ has minimal rank among residues stabilized by a subgroup commensurable with $H$. Note that since $H$ is infinite and $\Gamma$ acts properly on $X$ we necessarily have $I \neq \emptyset$.  Let $R'$ be another residue of type $I$. Since $\Gamma$ acts chamber transitively on $X$, there is $g \in \Gamma$ such that $R' = g(R)$. It follows that the subgroup $H ' = H \cap gHg^{-1}$, which has finite index in $H$, stabilizes both $R$ and $R'$. By the irreducibility assumption and since $I \neq S$, we have $I \cup I^\perp \neq S$. In particular by \cite[Proposition 2.7]{Cap-RAB} one can find $R'$ of type $I$ such that $R$ and $R'$ are not parallel. Since $H'$ stabilizes both of them, it also stabilizes $\mathrm{pr}_R(R')$, which is a residue of rank strictly less than the one of $R$. This contradicts the above minimality assumption.
	
	We consider the completion $\Gamma /\! \! / H$ and denote by $\tau_{\Gamma,H} : \Gamma  \to\Gamma /\! \! / H$ the homomorphism from $\Gamma$ to $\Gamma /\! \! / H$. We want to establish the:

	\textbf{Claim:} for every $s \in S$, there is a compact normal subgroup $K_s$ of $L$ such that $\tau_{\Gamma,H}(P_s) \leq K_s$. 
	
	Suppose that the claim is proven. Since the subgroup generated by finitely many compact normal subgroups is compact and $\Gamma$ is generated by the subgroups $(P_s)_{s \in S}$, it follows  that $\tau_{\Gamma,H}(\Gamma)$ is relatively compact in $\Gamma /\! \! / H$. But $\tau_{\Gamma,H}(\Gamma)$ being a dense subgroup of $\Gamma /\! \! / H$, we infer that $\Gamma /\! \! / H$ is compact. Since by definition  $\Gamma /\! \! / H$ admits a compact open subgroup with trivial normal core, we infer $\Gamma /\! \! / H$ is finite, and $H$ has finite index in $\Gamma$. The remainder of the proof is dedicated to proving the claim. 
	
	Fix $s \in S$. Clearly the statement of the proposition holds if $|S|=1$, so we can assume $|S| \geq 2$.  Since $(W,S)$ is irreducible, $s \cup s^\perp \subsetneq S$. Let $I_- = s \cup s^\perp $, $I_+ = S \setminus \left\lbrace s \right\rbrace $ and $I_0 = I_- \cap I_+ = s^\perp$. Consider the graph $T$ whose vertices are residues of $X$ of type $I_-$ or $I_+$, with an edge between a residue $R_-$ of type $I_-$ and a residue $R_+$ of type $I_+$ if $R_-$ and $R_+$ share a residue of type $I_0$. According to \cite[Lemma 4.3]{Hag-Pau-arbo}, the graph $T$ is a tree. The group $\Gamma$ acts on $T$ without inversion, with two orbits of vertices and one orbit of edges. Since $H$ is a commensurated subgroup of $\Gamma$, we deduce from \cite[Proposition 4.2]{LB-Wes-commens} that either $H$ fixes a vertex of $T$, or the $H$-action on $T$ is minimal. By the first part of this proof, we know that $H$ cannot stabilize any residue of type  $I_-$ or $I_+$. So it follows that $H$ acts minimally on $T$. Since $H$ is finitely generated, this implies that $H$ acts on $T$ with finitely many orbits of edges \cite[5.6]{Bass-Lub-book}.
	
	Now let $c_0$ be the chamber of $X$ corresponding to the identity element of $\Gamma$. Let $R$ be the residue of $X$ of type $s^\perp$ containing $c_0$, and let $\Gamma_R$ be its stabilizer in $\Gamma$. We can view $R$ as an edge of $T$, and by the above paragraph the $\Gamma$-orbit of $R$ is divided into finitely many $H$-orbits. Equivalently, $H \backslash \Gamma / \Gamma_R$ is finite. If $\gamma_1, \ldots, \gamma_n \in \Gamma$ are such that $\Gamma = \bigcup_i H \gamma_i \Gamma_R$, then $\Gamma /\! \! / H =\overline{ \tau_{\Gamma,H}(\Gamma)} =  \bigcup_i \overline{ \tau_{\Gamma,H}(H)} \tau_{\Gamma,H}(\gamma_i) \overline{\tau_{\Gamma,H}(B)}$. Since $\overline{ \tau_{\Gamma,H}(H)}$ is compact, $\overline{\tau_{\Gamma,H}(\Gamma_R)}$ is a cocompact subgroup of $\Gamma /\! \! / H$. Now $\Gamma_R$ is the subgroup of $\Gamma$ generated by the subgroups $P_{s'}$ when $s' $ ranges over $s ^\perp$, so $\Gamma_R$ commutes with $P_s$. Therefore the centraliser of $\tau_{\Gamma,H}(P_s)$ is a cocompact subgroup of $\Gamma /\! \! / H$. By Proposition \ref{prop-coc-cpt-normal} there must exist a compact normal subgroup $K_s$ of $\Gamma /\! \! / H$ such that $\tau_{\Gamma,H}(P_s) \leq K_s$, which concludes the proof of the claim.
\end{proof}

\begin{Theorem}[{Caprace \cite[Theorem~1.1]{Cap-RAB}}] \label{thm-RAB-simple}
	Let $X$ be a thick irreducible semi-regular right angled building of non-spherical type $(W,S)$. Then the group $\Aut_0(X)$ is abstractly simple. 
\end{Theorem}

\begin{Theorem} \label{thm-normal-subgroups-AutX}
	Let $X$ be a thick irreducible semi-regular right angled building of non-spherical type $(W,S)$, and write $G = \Aut(X)$. Then: \begin{enumerate} \item $\Res(G)$ is a cocompact normal subgroup of $G$ containing $\Aut_0(X)$; \item  $G$ has a maximal compact normal subgroup $K$, and $K \cap \Aut_0(X)$ is trivial; \item  every closed normal subgroup of $G$ is either contained in $K$, or contains $\Res(G)$. \end{enumerate}
\end{Theorem}

\begin{proof}
	By Theorem  \ref{thm-RAB-simple} the group $\Aut_0(X)$  is abstractly simple and non-discrete. The inclusion $\Aut_0(X) \leq \Res(G)$ follows. Let now $N$ be a closed normal subgroup of $G$, $U$ a compact open subgroup of $G$ and $O := NU$. According to \cite[Proposition 5.3]{DeMedts-Silva-RAB}, the group $\Aut_0(X)$ has the property that all its open subgroups are compactly generated. Since that property is inherited from a closed cocompact subgroup to an ambient group, the group $\Aut(X)$ has this property as well. In particular $O$ is compactly generated.  Since $N$ is normal and $U$ is commensurated in $G$, $O$ is also a commensurated subgroup of $G$. The subgroup $H = \Gamma_P \cap O$ is therefore commensurated in $\Gamma_P$, and finitely generated because it is cocompact in the compactly generated group $O$. By Proposition \ref{prop-inside-lattice} it follows that $H$ is finite or finite index in $\Gamma_P$. In the first case we infer $O$ is compact, and therefore $N$ is compact. In the second case we infer $O$ has finite index in $G$, and therefore $N$ is cocompact. Since any closed cocompact subgroup contains $\Res(G)$, we have shown every closed normal subgroup of $G$ is either compact or contains $\Res(G)$. When applied to the LE-radical $\RadLE(G)$, which cannot contain  $\Res(G)$ because $\Aut_0(X)$ is not locally elliptic, we infer $\RadLE(G)$ is compact. $K = \RadLE(G)$ is therefore  a maximal compact normal subgroup, and we have shown the conclusion. 
\end{proof}

\subsection{Local actions on spheres}

The goal of this subsection is to prove Proposition \ref{prop-RAB-loc-hom}, which provides the lack of local finite generation needed to apply Theorem~\ref{thm-main-not-loc-fg} in the present setting.

Let $c\in \mathrm{Ch}(X)$. For  $n \geq 1$, let $S(c,n)$ be the sphere around $c$ in the chamber graph, i..e.\ the set of chambers of $X$ at distance $n$ from $c$. We also let $B(c,n)$ be the set of chambers at distance at most $n$ from $c$. Let $U$ denote the stabilizer of $c$ in $\Aut(X)$, which is a compact open subgroup of $\Aut(X)$. The action of $U$ on $X$ stabilizes $S(c,n)$ for all $n$, and hence induces a continuous  homomorphism $U \to \mathrm{Sym}(S(c,n))$. By composing with the signature and considering the product over positive integers, we obtain a continuous homomorphism $\varphi : U \to \prod_{n \geq 1} \left\lbrace \pm 1\right\rbrace $. 

\begin{Proposition} \label{prop-RAB-loc-hom}
Let $c\in \mathrm{Ch}(X)$. Let $U$ denote the stabilizer of $c$ in $\Aut(X)$, and $U_0 = U \cap \Aut_0(X)$ . Then the continuous homomorphism $\varphi : U \to  \prod_{n \geq 1} \left\lbrace \pm 1\right\rbrace $ associated to the action of $U$ on spheres around $c$ is such that the restriction of $\varphi$ to $U_0$ is surjective. 
\end{Proposition}

The proof of the proposition relies on the following, which is a special case of \cite[Lemma 6.3]{Cap-RAB}. 

\begin{Lemma} \label{lem-Caprace-extends}
Let $X$ be a semi-regular right angled building of type $(W,S)$, and $c\in \mathrm{Ch}(X)$. Let $s \in S$, and $R$ a residue of type $s \cup s^\perp$. Let $c' = \mathrm{pr}_R(c)$, and let $\sigma$ be the panel of $c'$ of type $s$, and let $n$ be the distance from $c$ to $c'$. Let $T$ be a permutation of $ \mathrm{Ch}(\sigma)$ such that $T$ fixes $c'$. Then there exists $g \in \Aut_0(X)$ whose restriction to $\mathrm{Ch}(\sigma)$ is equal to $T$, and such that $g$ acts trivially on $B(c,n+1) \setminus \mathrm{Ch}(\sigma)$. 
\end{Lemma}

\begin{proof}[Proof of Proposition \ref{prop-RAB-loc-hom}]
Let $K = \prod_{n \geq 1} \left\lbrace \pm 1\right\rbrace $, and let $K_n$ be the kernel of the projection from $K$ to the first $n$ coordinates, and $F_n = K/K_n$. We have to show that $U_0 \to F_n$ is surjective for all $n \geq 1$. It is enough to show that the image of $U_0 \to F_n$ contains $(1,\ldots,1,-1)$ for every $n \geq 1$.

Since $(W,S)$ is not spherical, one can find $s \in S$ such that $s \cup s^\perp \subsetneq S$. So for every panel $\sigma$ of type $s$, the unique residue $R$ of type $s \cup s^\perp$ containing  $\sigma$ is a proper residue of $X$. It follows that given $n \geq 1$, one can find such a residue $R$ such that $c_n = \mathrm{pr}_R(c)$ is at distance $n$ from $c$. Let $\sigma_n$ be the panel of type $s$ of $c_n$. According to Lemma \ref{lem-Caprace-extends}, for every permutation $T$ of $\mathrm{Ch}(\sigma_n)$ such that $T$ fixes $c_n$, there exists $g \in \Aut_0(X)$ whose restriction to $\mathrm{Ch}(\sigma_n)$ is equal to $T$, and such that $g$ acts trivially on $B(c,n+1) \setminus \mathrm{Ch}(\sigma_n)$. Since $\sigma_n$ has $q_s \geq 3$ chambers, it is possible to take $T$ as above being a transposition. The associated $g$ belongs to $U_0$ and maps to $(1,\ldots,1,-1)$ in $F_{n+1}$.  
\end{proof}

\begin{Corollary} \label{cor-RAB-not-loc-fg}
	Let $X$ be a thick semi-regular right angled building of non-spherical type $(W,S)$, and write $G = \Aut(X)$. If $J$ is a closed subgroup of $\Aut(X)$ such that $\Aut_0(X) \leq J$, then $J$ is not locally finitely generated. 
\end{Corollary}

\begin{proof}
	Let $c\in \mathrm{Ch}(X)$, $U$ the stabilizer of $c$ in $\Aut(X)$, and $U_0 = U \cap \Aut_0(X)$. By
Proposition \ref{prop-RAB-loc-hom} any intermediate subgroup $U_0 \leq V \leq U$ surjects onto $\prod_{n \geq 1} \left\lbrace \pm 1\right\rbrace $.  This compact group is not topologically finitely generated, so $V$ inherits this property. Applying this to $V = H \cap U$ gives the conclusion. 
\end{proof}

\subsection{The proof of Theorem \ref{thm-intro-main-RAB}}

We will invoke the following result proven by Haglund \cite[Theorem 4.30]{Hag-08-GeoDe} and Kubena--Thomas \cite[Density Theorem]{KubTho}. 

\begin{Theorem} \label{thm-RAB-commens-dense}
Let $X$ be the semi-regular right angled building associated to a graph product of finite groups $\Gamma_P$. Then the commensurator of $\Gamma_P$ in $\Aut(X)$ is dense in $\Aut(X)$. 
\end{Theorem}

\begin{proof}[Proof of Theorem \ref{thm-intro-main-RAB}]
In case $X$ is of spherical type, then $\Aut(X)$ is a finite group and the statement is trivially true. So we assume that $X$ is of non-spherical type. We then have to see all the requirements of Theorem \ref{thm-main-not-loc-fg} are met. Density of the commensurator is provided by Theorem  \ref{thm-RAB-commens-dense}. The stabilizer of a chamber is a compact open subgroup of $\Aut(X)$ with trivial normal core. Theorem \ref{thm-normal-subgroups-AutX}  says every closed normal subgroup of $G$ is either compact or contains $\Res(G)$, and Corollary \ref{cor-RAB-not-loc-fg} ensures closed subgroups containing $\Res(G)$ are not locally finitely generated (because those contain $\Aut_0(X)$). Therefore all the assumptions of Theorem \ref{thm-main-not-loc-fg} are verified, and we deduce that any infinite finitely generated commensurated subgroup $H$ of $\Comm_G(\Gamma_P)$ is virtually contained in $\Gamma_P$. In other words that reduces the problem to the situation where $H$ is contained in $\Gamma_P$. The latter is treated  by Proposition \ref{prop-inside-lattice}. 
\end{proof}

\section{Hyperbolic virtually special groups}

In this section we consider the situation of hyperbolic groups that are virtually special  \cite{Haglund-Wise-GAFA-08}. Recall that by Agol's theorem this class encompasses all hyperbolic groups acting properly and cocompactly on a CAT(0) cube complex \cite{Agol-2013}. For instance this includes every hyperbolic  graph product of finite groups $\Gamma_P$ (and whether $\Gamma_P$ is hyperbolic is characterized by a simple condition on the defining graph \cite{Meier-96-graph-prod-hyp}). 

The following statement follows from works of Haglund--Wise and Wilton--Zalesskii. We are grateful to Henry Wilton and Pavel Zalesskii for a useful discussion about \cite{Wilton-Zalesskii-G&T}. 

\begin{Proposition} \label{prop-completion-hyp-special}
	Let $\Gamma$ be hyperbolic virtually special and not virtually cyclic. Then $\widehat{\Gamma}$ satisfies the conditions (\ref{item-many-primes}), (\ref{item-no-commuting-normal}), (\ref{item-no-finite-exponent-normal}) from Theorem \ref{thm-combined-arguments}. 
\end{Proposition}

\begin{proof}
	By \cite{Haglund-Wise-GAFA-08} there is a finite index subgroup $\Gamma'$ of $\Gamma$ and a surjective homomorphism $\Gamma' \to \Z$. This induces a surjective homomorphism $\widehat{\Gamma'} \to \widehat{\Z}$, and $\lpc(\widehat{\Gamma'})$ therefore consists of all primes. Since $\lpc(\widehat{\Gamma'}) =  \lpc(\widehat{\Gamma})$, (\ref{item-many-primes}) holds. The results of \cite{Wilton-Zalesskii-G&T} show that there is a finite index subgroup $\Gamma_0$ of $\Gamma$ such that every closed non-trivial normal subgroup of $\widehat{\Gamma_0}$ contains a non-abelian free profinite group. See Theorem 3.3, Lemma 7.3 and Proposition 6.6 in \cite{Wilton-Zalesskii-G&T}. This applies to $M \cap \widehat{\Gamma_0}$ provided $M$ is an infinite closed normal subgroup of $\widehat{\Gamma}$, and shows $\lpc(M)$ is infinite. As observed in Remark \ref{rmq-normal-lpc-infinite-sufficient}, this ensures (\ref{item-no-commuting-normal}) and (\ref{item-no-finite-exponent-normal}).
\end{proof}

\begin{Corollary} \label{cor-hyp-special-CSP}
	Let $\Gamma$ be a finitely generated cocompact lattice of a tdlc group $G$, and $C$ a dense subgroup of $G$ such that $\Gamma \leq C \leq \Comm_G(\Gamma)$. Suppose $\Gamma$ is hyperbolic virtually special and not virtually cyclic, and $\Gamma$ has the CSP in $C$. Suppose also $\QZ(G)$ is discrete, $\Res(G)$ is open, and every abstract normal subgroup of $G$ is contained in $\QZ(G)$ or contains $\Res(G)$. Then every finitely generated commensurated subgroup of $C$ is virtually contained in $\Gamma$.
\end{Corollary}

\begin{proof}
Follows from Theorem \ref{thm-combined-arguments} and Proposition \ref{prop-completion-hyp-special}. 
\end{proof}

\bibliographystyle{amsalpha}
\bibliography{commensurator}

\end{document}